\documentclass[11pt,reqno]{amsart}
\usepackage{color}
\usepackage[active]{srcltx}
\usepackage{a4wide}
\usepackage{amssymb, amsmath, mathtools}
\usepackage{graphicx}
\usepackage{float}
\usepackage[active]{srcltx}
\usepackage[ansinew]{inputenc}
\usepackage{hyperref}

\theoremstyle{plain}
\newtheorem{theorem}{Theorem}[section]
\newtheorem{maintheorem}{Theorem}
\newtheorem{lemma}[theorem]{Lemma}
\newtheorem{proposition}[theorem]{Proposition}

\newtheorem{maincorollary}[maintheorem]{Corollary}

\theoremstyle{remark}

\newtheorem{definition}{Definition}

\newtheorem{remark}[theorem]{Remark}

\numberwithin{equation}{section}

\newcommand{\NN}{{\mathbb{N}}}
\newcommand{\ZZ}{{\mathbb{Z}}}
\newcommand{\RR}{{\mathbb{R}}}
\newcommand{\EU}{{\mathbb{S}}}

\newcommand{\In}{{\text{In}}}
\newcommand{\Out}{{\text{Out}}}
\newcommand{\Fix}{{\text{Fix}}}
\newcommand{\loc}{{\text{loc}}}

\newcommand{\dpt}{\displaystyle}

\begin{document}

\title[From an attracting torus to strange attractors]{Unfolding a Bykov attractor: \\ from an attracting torus to strange attractors}
\author[Alexandre Rodrigues]{Alexandre A. P. Rodrigues \\ Centro de Matem\'atica da Univ. do Porto \\ Rua do Campo Alegre, 687,  4169-007 Porto,  Portugal }
\address{Alexandre Rodrigues \\ Centro de Matem\'atica da Univ. do Porto \\ Rua do Campo Alegre, 687 \\ 4169-007 Porto \\ Portugal}
\email{alexandre.rodrigues@fc.up.pt}

\date{\today}

\thanks{AR was partially supported by CMUP , which is financed by national funds through FCT -- Fundação para a Ci\^encia e Tecnologia, I.P., under the project with reference UIDB/00144/2020.  The author also acknowledges financial support from Program INVESTIGADOR FCT (IF/00107/2015).}

\subjclass[2010]{ 34C28; 34C37; 37D05; 37D45; 37G35 \\
\emph{Keywords:} Bykov attractor; Heteroclinic cycle; Torus-breakdown; Strange attractors, Hopf-zero singularity.}

\begin{abstract}
In this paper we present a comprehensive mechanism for the emergence of strange attractors in a two-parametric family of differential equations acting on a three-dimensional sphere. When both parameters are zero, its flow exhibits an attracting heteroclinic network (Bykov network) made by two 1-dimensional connections and one 2-dimensional separatrix between two hyperbolic saddles-foci with different Morse indices.  After slightly increasing both parameters, while keeping the one-dimensional connections unaltered, we focus our attention in the case where the two-dimensional invariant manifolds of the equilibria do not intersect. 

Under some conditions on the parameters and on the eigenvalues of the linearisation of the vector field at the saddle-foci, we prove the existence of many complicated dynamical objects, ranging from an attracting quasi-periodic torus  to H\'enon-like strange attractors, as a consequence of the \emph{Torus-Breakdown} Theory. The mechanism for the creation of horseshoes and strange attractors is also discussed. Theoretical results are applied to show the occurrence of strange attractors in some analytic unfoldings of a Hopf-zero singularity. 
\end{abstract}

\maketitle
\setcounter{tocdepth}{1}

\section{Introduction}\label{intro}

\subsection{Strange attractors}
Many aspects contribute to the richness and complexity of a dynamical system. One of them is the existence of strange attractors. Before going further, we introduce the following notion (adapted to the situation under consideration):

\begin{definition}
A  (H\'enon type) \emph{strange attractor} of a two-dimensional dissipative diffeomorphism $R$ defined in a Riemannian manifold, is a compact invariant set $\Lambda$ with the following properties:
\begin{itemize}
\item $\Lambda$  equals the closure of the unstable manifold of a hyperbolic periodic point;
\item the basin of attraction of $\Lambda$   contains an open set (and thus has positive Lebesgue measure);
\item there is a dense orbit in $\Lambda$ with a positive Lyapounov exponent (exponential growth of the derivative along its orbit);
\item $\Lambda$ is not hyperbolic.
\end{itemize}
A vector field possesses a strange attractor if the first return map to a cross section does.  
\end{definition}
The rigorous proof of the strange character of an invariant set is a great challenge and the proof of the persistence (in measure) of such attractors is a very involving task. 
Based on \cite{BC91}, Mora and Viana \cite{MV93} proved the emergence of strange attractors in the process of creation or destruction of the Smale horseshoes that appear through a bifurcation of a tangential homoclinic point.
In the present paper, rather than exhibit  the existence of strange attractors, we explore a mechanism to  
guarantee the existence of such complex dynamics in the unfolding of a Bykov attractor, an expected phenomenon  close to a $\mathbb{SO}(2)$-equivariant system. In the present paper, the abundance of strange attractors is a consequence of the  \emph{Torus-breakdown Theory} developed in \cite{AS91, AHL2001, AH2002, Aronson}. See also \cite{WY}.

\subsection{The object of study}
Our starting point is a two-parametric 
 differential equation  $\dot{x}=f_{(A, \lambda)}(x)$ defined in the three-dimensional sphere $\EU^3$
with two saddle-foci sharing  all the invariant manifolds of dimensions one and two for $A=\lambda=0$,
forming an attracting heteroclinic network $\Gamma$ with a non-empty basin of attraction $\mathcal{U}$.
 We study the global transition of the dynamics from $\dot{x}=f_{(0,0)}(x)$ to a  smooth two-parameter family $\dot{x}=f_{(A,\lambda)}(x)$ that  breaks  the network (or part of it). 
Note that, for small perturbations, the set $\mathcal{U}$ is still positively invariant. When $A, \lambda\neq 0$, we assume that the one-dimensional connections persist, and the 
two dimensional invariant manifolds are generically transverse (either intersecting or not). 

When $\lambda>A\geq 0$, the two-dimensional invariant manifolds meet transversely, giving rise to a complex network, that consists of a union of Bykov cycles \cite{Bykov00}. The dynamics in the maximal invariant set contained in $\mathcal{U}$, contains, but does not coincide with, the suspension of horseshoes accumulating on the heteroclinic network described in \cite{ACL05, KLW, LR, Rodrigues2, Rodrigues3}. In addition, close to the organising  center, the flow contains infinitely many heteroclinic tangencies and attracting limit cycles with long  periods, coexisting with sets with positive entropy, giving rise the so called \emph{quasi-stochastic attractors}. All dynamical models with quasi-stochastic attractors were found, either analytically or  by computer simulations, to have tangencies of invariant manifolds  \cite{Gonchenko97}. 
   Recently, there has been a renewal of interest of this type of heteroclinic bifurcation in the reversible \cite{DIKS, KLW, Lamb2005}, equivariant  \cite{LR, Rodrigues2, Rodrigues3} and conservative \cite{BessaRodrigues} contexts.

\medbreak

\textbf{The novelty:} 
The case $A>\lambda\geq 0$  corresponds to the situation where the two-dimensional invariant manifolds do not intersect, which 
has not been yet studied.  Although the network $\Gamma$ associated to the equilibria is destroyed,  complex dynamics appears near the ghost of it. 
In the present article, it is shown that the perturbed system may manifest regular behavior corresponding to the existence of a smooth invariant torus, and may also have chaotic regimes. In the region of transition from regular behavior (attracting torus) to chaotic dynamics (suspended horseshoes), using known results about \emph{Arnold tongues}, we prove the existence of lines with homoclinic tangencies to dissipative periodic solutions, responsible for the existence of persistent strange attractors nearby. This  phenomenon  has already been  observed by \cite{Wang2, WY} in the context of non-autonomous differential equations. Our theoretical results may be applied in specific unfoldings of the Hopf-zero singularity  (Case III of \cite{GH}).

\subsection{This article}

We study the dynamics arising near  a differential equation with a specific attracting heteroclinic network (Bykov attractor).
We show that, when a two-dimensional  connection is broken with a  prescribed configuration, the dynamics
 undergoes  a global transition from regular to chaotic dynamics.

We discuss the  global bifurcations  that occur as the parameters $(A,\lambda)$ vary. 
We complete our results by reducing our problem to that of a first return map having an attracting (non-contractible) curve and we study its generic bifurcations.
As we will see in Section \ref{torus_bif}, the mechanism of the transition from regular dynamics to chaotic dynamics is not so standard as in \cite{PT, YA}, where just saddle-node, periodic doubling and Newhouse bifurcations were involved in the  the annihilation of hyperbolic horseshoes.

This article is organised as follows. 
In Section~\ref{s:setting}, after some basic definitions, we describe precisely our object of study and in Section \ref{se:state of art}
we review some literature related to it.
In Section~\ref{main results} we state the main results of the article. 
 The coordinates and other notation  used in the rest of the article are presented in  Section~\ref{localdyn}. 
In Sections \ref{proof Th A}, \ref{Prova Th B} and \ref{proof Th E}, we prove the main results of the manuscript. 
In Section \ref{torus_bif},   we describe generic ways to break an attracting two-dimensional torus, which involve homoclinic tangencies produced by the stable and unstable manifolds of a dissipative saddle. These tangencies are the origin of persistent strange attractors.
This mechanism is discussed in Section \ref{Hopf} for generic unfoldings of the Hopf-zero singularity (Case III of \cite{GH}). For the reader's convenience, we have compiled at the end of the paper a list of definitions in a short glossary.


\medbreak
Throughout this paper, we have endeavoured to make a self contained exposition bringing together all topics related to the proofs. We have stated short lemmas and we have drawn illustrative figures to make the paper easily readable.

\section{Setting}
\label{s:setting}
We will enumerate the main assumptions concerning the configuration of the network and the intersection of the invariant manifolds of the equilibria.   We refer the reader to Appendix \ref{Definitions} for precise definitions.

\subsection{The organising center}
For $\varepsilon>0$ small, consider the two-parameter family of $C^3$-smooth differential equations
\begin{equation}
\label{general2.1}
\dot{x}=f_{(A, \lambda)}(x)\qquad x\in \EU^3 \qquad A, \lambda \in [0, \varepsilon] 
\end{equation}
 where $\EU^3$ denotes the unit three-sphere, endowed with the usual topology. Denote by $\varphi_{(A, \lambda)}(t,x)$, $t \in \RR$, the associated flow\footnote{Since $\EU^3$ is a compact set without boundary, the local solutions of \eqref{general2.1} could be extended to $\RR$.}, satisfying the following hypotheses for $A=\lambda=0$:

\bigbreak
\begin{enumerate}
 \item[\textbf{(P1)}] \label{B1}  There are two different equilibria, say $O_1$ and $O_2$.
 \bigbreak
 \item[\textbf{(P2)}] \label{B2} The spectrum of $df_x$ is:
 \medbreak
 \begin{enumerate}
 \item[\textbf{(P2a)}] $E_1$ and $ -C_1\pm \omega_1 i $ where $C_1>E_1>0, \quad \omega_1>0$, \qquad for $x=O_1$;
 \medbreak
 \item[\textbf{(P2b)}] $-C_2$ and $ E_2\pm \omega_2 i $ where $C_2> E_2>0, \quad \omega_2>0$,  \qquad for $x=O_2$.
 \end{enumerate}
 \end{enumerate}
\bigbreak 
 Thus the equilibrium $O_1$ possesses a 2-dimen\-sional stable and $1$-dimen\-sional unstable manifold and the equilibrium $O_2$ possesses a 1-dimen\-sional stable and $2$-dimen\-sional unstable manifold. For $M\subset \EU^3$, denoting by $\overline{M}$ the topological closure of $M$, we also assume that:
 \begin{enumerate}
 \bigbreak
  \item[\textbf{(P3)}]\label{B3} The manifolds $\overline{W^u(O_2)}$ and $\overline{ W^s(O_1)}$ coincide and   $\overline{W^u(O_2)\cap W^s(O_1)}$ consists of a two-sphere (also called the $2D$-connection).
  \end{enumerate} \bigbreak
  and 
   \bigbreak
   \begin{enumerate}
\item[\textbf{(P4)}]\label{B4} There are two trajectories, say  $\gamma_1, \gamma_2$, contained in  $W^u(O_1)\cap W^s(O_2)$, one in each connected component of $\EU^3\backslash \overline{W^u(O_2)}$ (also called the $1D$-connections).
\end{enumerate}
 
 \begin{figure}[h]
\begin{center}
\includegraphics[height=7cm]{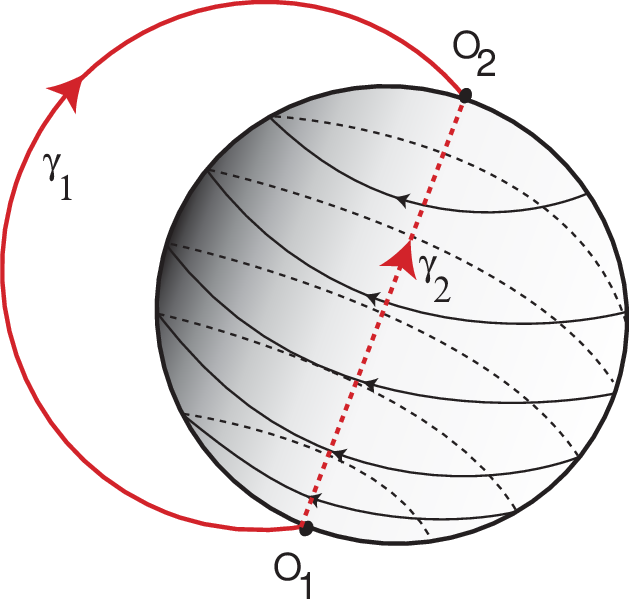}
\end{center}
\caption{\small  Sketch of the Bykov attractor $\Gamma$ satisfying \textbf{(P1)--(P4)}. }
\label{Bykov1}
\end{figure}

\bigbreak
The two equilibria $O_1$ and $O_2$, the two-dimensional heteroclinic connection from $O_2$ to $O_1$ refered in \textbf{(P3)} and the two trajectories listed in  \textbf{(P4)}  build a heteroclinic network we will denote hereafter by $\Gamma$. This network has two cycles and is illustrated in Figure \ref{Bykov1}.   This set has an \emph{attracting} character \cite{LR2015}, this is why it will be called a \emph{Bykov\footnote{The terminology \emph{Bykov} is a tribute to V. Bykov who has dedicated his latest research activity to heteroclinic cycles with similar properties to those of $\Gamma$.} attractor}. 

\begin{lemma}[\cite{LR2015}]
\label{attractor_lemma}
The set $\Gamma$ is asymptotically stable.
\end{lemma}

Therefore, we may find an open neighborhood $\mathcal{U}$ of the heteroclinic network $\Gamma$ having its boundary transverse to the flow of $\dot{x}=f_{(0,0)}(x)$ and such that every solution starting in $\mathcal{U}$ remains in it for all positive time and is forward asymptotic to $\Gamma$.
\subsection{Chirality}
There are two different possibilities for the geometry of the flow around $\Gamma$,
depending on the direction in which trajectories turn around the one-dimensional heteroclinic connection from $O_1$ to $O_2$.
To make this rigorous, we need some new concepts. 

\medbreak
Let $V_1$
and $V_2$ be small disjoint neighborhoods of $O_1$ and $O_2$ with disjoint boundaries $\partial V_1$ and $\partial V_2$, respectively. These neighborhoods will be constructed with detail in Section~\ref{localdyn}.
Trajectories starting at $\partial V_1\backslash W^s(O_1)$ near $W^s(O_1)$ go into the interior of $V_1$ in positive time, then follow the connection from $O_1$ to $O_2$, go inside $V_2$, and then come out at $\partial V_2$
as in Figure~\ref{Chirality}. Let $\mathcal{Q}$ be a piece of trajectory like the one which has been constructed from $\partial V_1$ to $\partial V_2$.
Now join its starting point to its end point by a line segment as in Figure~\ref{Chirality}, forming a closed curve, that we call the  \emph{loop} of $\mathcal{Q}$.
The loop of $\mathcal{Q}$ and the network $\Gamma$ are disjoint closed sets. 

\begin{definition}(\cite{LR2015})
We say that the two saddle-foci $O_1$ and $O_2$ in $\Gamma$ have the  \emph{same chirality} if the loop of every trajectory (starting near $O_1$) is linked to $\Gamma$ in the sense that the two closed sets cannot be disconnected by an isotopy. Otherwise, we say that $O_1$ and $O_2$ have \emph{different chirality}.
\end{definition}
 Our next assumption is topological and may be written as:

\medbreak

\begin{enumerate}
\item[\textbf{(P5)}] \label{B5} The saddle-foci $O_1$ and $O_2$ have the same chirality.
\end{enumerate}

\begin{figure}[h]
\begin{center}
\includegraphics[height=7cm]{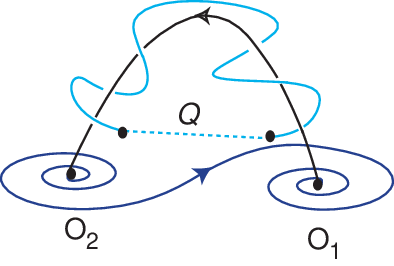}
\end{center}
\caption{\small  Illustration of Property  \textbf{(P5)}: the saddle-foci $O_1$ and $O_2$ have the same chirality.}
\label{Chirality}
\end{figure}
\bigbreak
For  $r \geq 3$, denote by  $\mathfrak{X}^r(\EU^3)$, the set of $C^r$ two-parameter family of vector fields on $\EU^3$ satisfying Properties \textbf{(P1)--(P5)}, endowed with the $C^r$-topology.
\bigbreak
\subsection{Perturbing terms}
With respect to the effect of the two parameters $A$ and $\lambda$ on the dynamics, we assume that:
 \medbreak
\begin{enumerate}
\item[\textbf{(P6)}] \label{B5} For $A> \lambda \geq 0$, the two trajectories within $W^u(O_1)\cap W^s(O_2)$ persist.
\end{enumerate}
 \medbreak
By Kupka-Smale Theorem, generically the invariant two-dimensional manifolds $W^u (O_2)$ and $W^s (O_1)$ are  transverse (intersecting or not). Throughout this article, we assume that:
 \medbreak
\begin{enumerate}
\item[\textbf{(P7a)}]\label{B7} For $A> \lambda \geq 0$, the two-dimensional manifolds $W^u(O_2)$ and $W^s(O_1)$ do not intersect.
\end{enumerate}

 \medbreak

 and
\medbreak

\begin{enumerate}
\item[\textbf{(P8)}] \label{B8} The transitions along the connections $[O_1 \rightarrow  O_2]$ and $[O_2 \rightarrow  O_1]$ are given, in local coordinates, by the \emph{Identity map} and by  $$(x,y)\mapsto (x,y+A + \lambda \Phi(x))$$ respectively, where $\Phi:\EU^1 \rightarrow \EU^1$ is a Morse function with at least two non-degenerate critical points ($\EU^1=\RR \pmod{2\pi}$). This assumption will be detailed later in Section~\ref{localdyn}.
\end{enumerate}
\medbreak

 For the moment, without loss of generality,  let us also assume that  $\Phi(x)=\sin x$,  $x\in \EU^1$, which has exactly two critical points.

\subsection{Constants}
\label{constants1}
For future use, we settle the following notation:
\begin{equation}
\label{constants}
\delta_1 = \frac{C_1}{E_1 }>1, \qquad \delta_2 = \frac{C_2}{E_2 }>1, \qquad \delta=\delta_1\, \delta_2>1 
 \end{equation}
and
\begin{equation}
\label{constants2}
K = \frac{E_2 \, +C_1\, }{E_1E_2}>0, \qquad   K_\omega= \frac{E_2 \, \omega_1+C_1\, \omega_2 }{E_1E_2}>0 \qquad \text{and} \qquad a=\frac{\lambda}{A}.
 \end{equation}

\medbreak

\section{Overview}\label{se:state of art}
In this section, we review some known results about the bifurcation of codimension 2 under consideration, which are summarized in Table 1 and illustrated in Figure \ref{cross_sections2c}.

\subsection{Case 1: $A=\lambda=0$}
The network $\Gamma$ is asymptotically stable (see Lemma \ref{attractor_lemma}). The two-dimensional manifolds $W^u(O_2)$ and $W^s(O_1)$ coincide and the global attractor of $f_{(0,0)}$
is made of the equilibria $O_1, O_2$, and the two trajectories of $[O_1 \rightarrow  O_2]$ together with a sphere which is both the stable manifold of $O_1$ and the unstable manifold of $O_2$.  See Figure \ref{Bykov1}.

This attractor, the so called \emph{Bykov attractor}, has a finite number of moduli of stability and the points of its proper basin of attraction have historic behavior \cite{CR2018}. The coincidence of the two-dimensional invariant manifolds of $O_1$ and $O_2$ prevents visits to both cycles.

\begin{table}[htb]
\begin{center}
\begin{tabular}{|c|c|c|c|} \hline 
Case(s) &\emph{Parameters}  & $\lambda=0$ & $\lambda>0$  \\
\hline \hline
& && \\
1 and 2&$A=0$ & Attracting network & Horseshoes  + strange attractors \cite{LR2016}  \\  && \cite{LR, LR2016}& + homoclinic tangencies \cite{LR2016, RodLab} \\ &&&\\ \hline
\hline 
&&& \\
3&$0< A< \lambda$ & --- & Horseshoes  + strange attractors \cite{LR2016} \\ &&& + homoclinic tangencies \cite{LR2016, RodLab} \\&&&\\ \hline \hline
&&& \\
4& &  & Torus \\& $A>\lambda$&Attracting torus& or horseshoes + strange attractors \\(New) &&& {(additional parameter: $K_\omega$)} \\ &&&\\  
\hline

\end{tabular}
\end{center}
\label{notation2}
\bigskip
\caption{Overview of the results.}
\end{table} 

\subsection{Case 2: $A=0$ and $\lambda>0$}
Let $f_{(A, \lambda)}\in \mathfrak{X}^r(\EU^3)$ be a two-parameter family of vector fields satisfying \textbf{(P1)--(P6)} and for which Property \textbf{(P7a)} does not hold. Trajectories within $W^u(O_1)\cap W^s(O_2)$ are kept but the manifolds $W^u(O_2)$ and $W^s(O_1)$ rearrange themselves, creating either transversal or tangential intersections. This upheaval causes an explosion of the non-wandering set of the flow, bringing forth a countable union of suspended horseshoes inside the set of trajectories that remain for all positive times in $\mathcal{U}$. These horseshoes accumulate at the stable/unstable manifolds of the equilibria (cf. \cite{Bykov00, LR}).  Furthermore, for a sequence of positive parameters arbitrarily close to zero, say $(\lambda_j)_j$, the flow associated to  $f_{(0, \lambda_j)}$ exhibits homo and heteroclinic tangencies, sinks with long periods and strange attractors  (cf. \cite{LR2016, MV93, Newhouse74}).

\begin{figure}[h]
\begin{center}
\includegraphics[height=7.6cm]{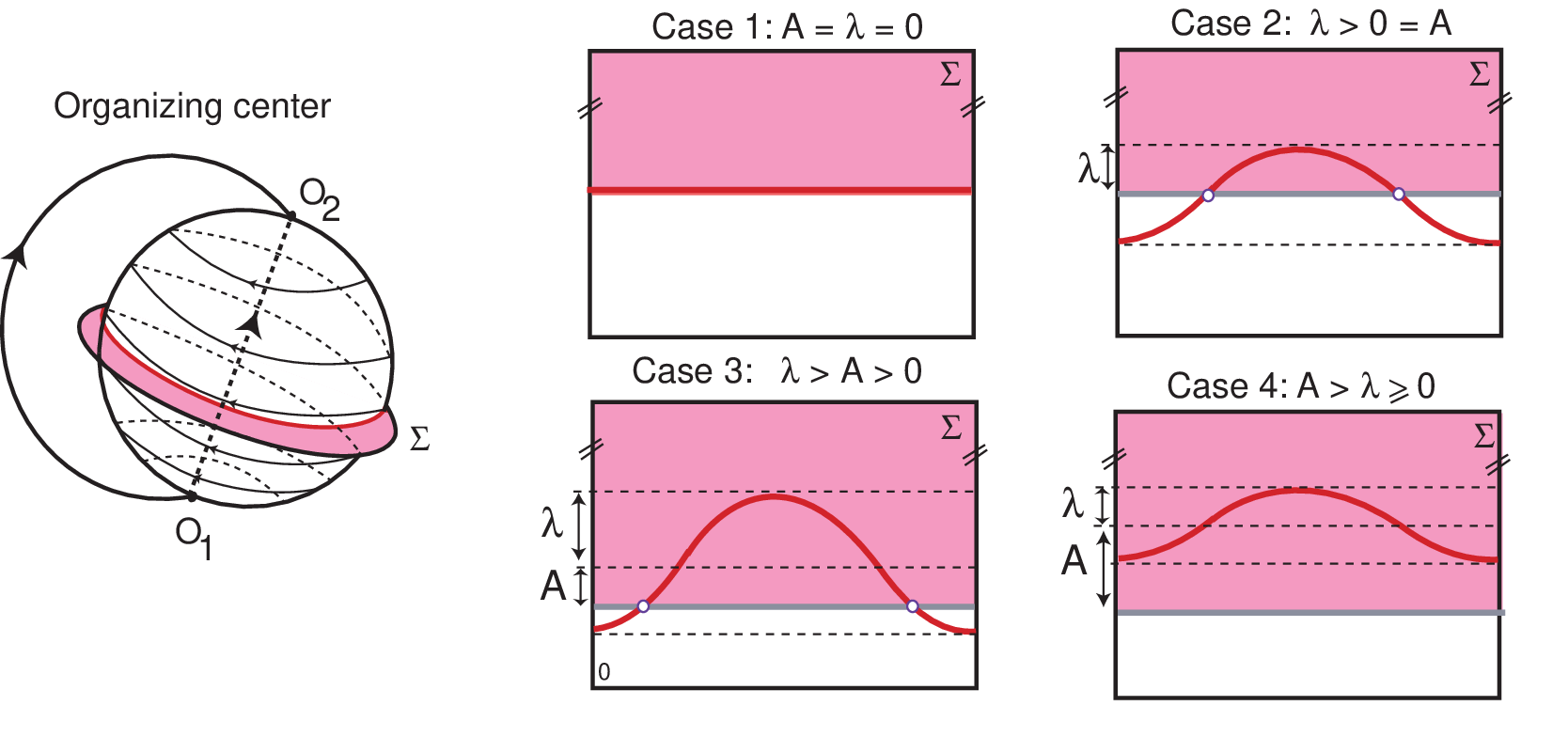}
\end{center}
\caption{\small Shape of the first hit of $W^u(O_2)$/$W^s(O_1)$ to $\Sigma$. The set $\Sigma$ a cross section to the Bykov attractor $\Gamma$ near the connection $W^u(O_2) \cap W^s(O_1)$. The symbol $\circ$ represents (transverse) heteroclinic connections from $O_2$ to $O_1$; the red line is the graph of $\Phi(x)$ representing $W^u (O_2)\cap \Sigma$;  the grey line is  $W^s (O_1)\cap \Sigma$; the double bars mean that the sides are identified.  }
\label{cross_sections2c}
\end{figure}

\subsection{Case 3:  $\lambda>A>0$ } In this case,  the manifolds $W^u(O_2)$ and $W^s(O_1)$ still meet transversely along at least two connections and thus the results of  \cite{LR, LR2016, Rodrigues3} still hold.  The transverse intersection of the two-dimensional invariant manifolds of the two equilibria  implies that the set of  trajectories that remain for all time in a small neighborhood of the Bykov cycle contains a locally-maximal hyperbolic set admitting a complete description in terms of symbolic dynamics, reminiscent of the results of \cite{Gonchenko97, Shilnikov70}.  An obstacle to the global symbolic description  of these trajectories
is the existence of tangencies that lead to the birth of stable periodic solutions, as described in \cite{GavS,  Gonchenko97, Newhouse79}.

\subsection{Case 4: $A>\lambda \geq 0$}
As far as we know, the cases $A>\lambda=0$ and $A>\lambda>0$ have not been studied. The goal of this article is to explore these cases where there are no cycles associated to $O_1$ and $O_2$. In contrast to Cases 2 and 3, where the involved bifurcations are related to horseshoe destruction, tangencies and Newhouse phenomena \cite{LR2016, PT}, in Case 4 the  bifurcations are connected with the \emph{Torus-Breakdown Theory} \cite{AS91, AHL2001, AH2002, Aronson}. Hence, hereafter, we also assume that:
\medbreak
\begin{enumerate}
\item[\textbf{(P7b)}] \label{B9}  $A>\lambda\geq 0$ (or, equivalently, $a =\frac{\lambda}{A}\in  \,[\, 0,1\, [\, $).
\end{enumerate}
\bigbreak

Under hypothesis \textbf{(P8)}, one has  \textbf{(P7a)}  $\Leftrightarrow$ \textbf{(P7b)}, as noticed in Remark \ref{equivalence}.  When we refer to \textbf{(P7)} we refer one of the above.
\medbreak

For $r \geq 3$, we denote by $\mathfrak{X}^r_{\text{Byk}}(\EU^3)\subset \mathfrak{X}^r(\EU^3)$, the set of $C^r$ two-parameter families of vector fields on $\EU^3$ satisfying conditions  \textbf{(P1)}--\textbf{(P8)}. An explicit example with a two parameter polynomial differential equation satisfying \textbf{(P1)}--\textbf{(P7)} has been given in  \S 4.1.3.2 of \cite{Aguiar_tese}.

\subsection*{Digestive remark}
Constants $A$ and $\lambda$ may be interpreted as follows: as depicted in Figure \ref{cross_sections2c}, the first hit of $W^u(O_2)$ and $W^s(O_1)$ to $\Sigma$ are two closed curves. In Case 4,  according to our model, the distance between the two curves  can be written as $A + \lambda \Phi(x)$, $x\in \EU^1$,  which may be seen as an  approximation of the Melnikov function associated to the two-dimensional invariant manifolds \cite[\S 4.5]{GH}.  The constant $A$ gives the averaged distance between the first hit of $W^u(O_2)$ and $W^s(O_1)$ in $\Sigma$; the parameter $\lambda$ describes fluctuations of $W^u(O_2)$ in $\Sigma$.

\section{Main results}\label{main results}

Let $\mathcal{T}$ be a neighborhood of the Bykov attractor $\Gamma$, which exists for $A=\lambda=0$.  For $\varepsilon>0$ small, let $\left(f_{(A, \lambda)}\right)_{(A,\lambda)\in [0, \varepsilon] ^2}$ be a two-parameter family of vector fields in $\mathfrak{X}_{Byk}^3(\EU^3)$ satisfying conditions \textbf{(P1)--(P8)}. 
\bigbreak
\begin{maintheorem}\label{thm:0}
Let $f_{(A, \lambda)} \in\mathfrak{X}_{Byk}^3(\EU^3)$. 
Then, there is $\tilde\varepsilon>0$ (small) such that the first return map to a given cross section to $\Gamma$ may be written (in local coordinates)   by:
$$
\mathcal{F}_{(A, \lambda)}(x, y)=\left[ x - K_\omega  \ln ( y+A+\lambda\sin x) \pmod{2\pi}  , \, \, ( y+A+\lambda\sin x)^\delta\right] + \ldots
$$
where $$(x,y)\in \mathcal{D}=\{x\in \RR \pmod{2\pi}, \quad y/\tilde \varepsilon  \in [-1, 1] \quad \text{and} \quad y + A + \lambda \sin x> 0\}$$ and the ellipsis stand for asymptotically small  terms depending on $x$ and $y$ which converge to zero along their derivatives.

\end{maintheorem}
\bigbreak

The proof of Theorem \ref{thm:0} is done in Section \ref{proof Th A} by composing local and transition maps. 
Since $\delta>1$, for $A$ small enough, the second component of $\mathcal{F}_{(A, \lambda)}$ is contracting and, under an additional hypothesis,  the dynamics of $\mathcal{F}_{(A, \lambda)}$ is dominated by the family of circle maps.

\bigbreak
\begin{maintheorem}\label{thm:B}
Let $f_{(A, \lambda)} \in\mathfrak{X}_{Byk}^3(\EU^3)$.
For $A>0$ small enough, if $\frac{\lambda}{A}<\frac{1}{\sqrt{1+K_\omega^2}}$, then there is an invariant closed curve $\mathcal{C}\subset \mathcal{D}$ as the maximal attractor for $\mathcal{F}_{(A, \lambda)}$.  This closed curve is not contractible on $\mathcal{D}$.
\end{maintheorem}
\bigbreak
The proof of Theorem is performed in Section \ref{Prova Th B} using the Afraimovich's Annulus Principle~\cite{AHL2001}. 
The  curve $\mathcal{C}$ is \emph{globally attracting} in the sense that, for every $p\in \mathcal{D}$, there exist a point $p_0\in \mathcal{C}$ such that 
$$
\lim_{n\rightarrow +\infty} \left| \mathcal{F}_{(A, \lambda)}^n (p) -\mathcal{F}_{(A, \lambda)}^n (p_0) \right| =0.
$$
The attracting invariant curve for  $\mathcal{F}_{(A, \lambda)}$ is  the graph of a smooth map and corresponds to an attracting two-torus for the flow of \eqref{general2.1}. Thus, in particular:

\bigbreak
\begin{maincorollary}\label{cor:D}
Let $f_{(A, \lambda)} \in\mathfrak{X}_{Byk}^3(\EU^3)$.
For $A>0$ small enough, if either  $K_\omega>0$ or $\lambda>0$ are sufficiently small, then:
\begin{itemize}
\item[\textbf{(a)}] there is a persistent two-dimensional torus which is globally attracting.  
\medbreak
\item[\textbf{(b)}] the dynamics of $\mathcal{F}_{(A, \lambda)}$ induces on  $\mathcal{C}$  a circle map. In this case, for any given interval of unit length $I$, there is a positive measure set $\Delta \subset I$ so that the rotation number of $\mathcal{F}_{(A, \lambda)}|_\mathcal{C}$ is irrational  if and only if $a=\lambda/A \in \Delta$.
\end{itemize}
\end{maincorollary}
 \bigbreak
 \begin{proof}
 \begin{itemize}
\item[\textbf{(a)}]  The existence of an invariant torus is a direct corollary of Theorem \ref{thm:B}. Since the torus is normally hyperbolic, its persistence follows from the theory for normally hyperbolic manifolds developed by Hirsch \emph{et al} \cite{HPS}.
\medbreak
\item[\textbf{(b)}] The first part of this item follows from \textbf{(a)}; the second part, concerning the rotation number on circle maps, may be found in Boyland \cite{Boyland} and Herman \cite{Herman}.  

\end{itemize}
 
 \end{proof}

The previous result implies the existence of a set of positive Lebesgue measure (in the bifurcation parameter $(A, \lambda)$)  for which the torus has a dense orbit, \emph{i.e.} the whole torus is a minimal attractor.

\begin{figure}[h]
\begin{center}
\includegraphics[height=7cm]{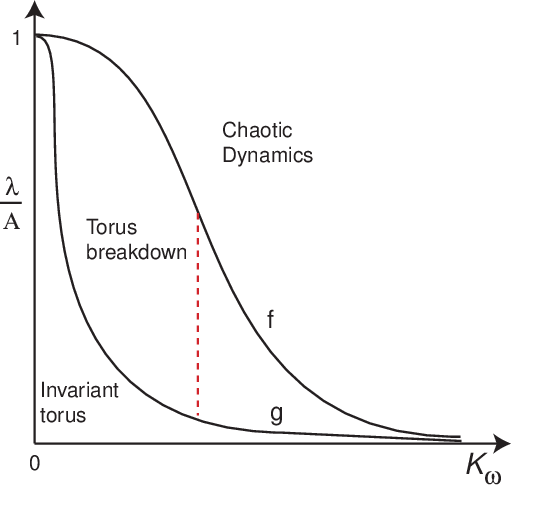}
\end{center}
\caption{\small Invariant attracting 2-torus and chaotic regions for the map $\mathcal{F}_{(A, \lambda)}$ with respect to the parameters $K_\omega$ and $\frac{\lambda}{A}$. }
\label{graph1}
\end{figure}
 \bigbreak
 
\begin{maintheorem}\label{thm:E}
Let $f_{(A, \lambda)} \in\mathfrak{X}_{Byk}^3(\EU^3)$.
For $A>0$ small enough, if  $$  \frac{\exp\left(\frac{6\pi}{K_\omega \, }\right)-1}{\exp\left(\frac{6\pi}{K_\omega \, }\right)-1/6}< \frac{\lambda}{A} <1,$$ then there exists a hyperbolic invariant closed subset $\Lambda $ in a  cross section to $\Gamma$ such that $\mathcal{F}_{(A, \lambda)}|_\Lambda$ is topologically conjugate to the Bernoulli shift of two symbols. 
\end{maintheorem}
\bigbreak
\begin{remark}
Under a stronger hypothesis,  the conclusion of Theorem \ref{thm:E} may be rephrased  as: $\mathcal{F}_{(A, \lambda)}|_\Lambda$ is topologically conjugate to the Bernoulli shift of $m$ symbols, with $2 \leq m \in \NN$, giving rise to \emph{rotational horseshoes}, a particular type of horseshoes characterized by Passegi \emph{et al} \cite{PPS} -- see Lemma \ref{two turns}.
\end{remark}

As a corollary of Theorems \ref{thm:B} and  \ref{thm:E}, in the bifurcation diagram $(a, K_\omega)=\left(\frac{\lambda}{A}, K_\omega \right)$, we may draw, in the first quadrant,  two smooth curves, the graphs of $g$ and $f$ in Figure \ref{graph1}, such that:
\begin{enumerate}
\medbreak
\item  $g(K_\omega)=\frac{1}{\sqrt{1+K_\omega^2}}$ and $f(K_\omega)=  \frac{\exp\left(\frac{6\pi}{K_\omega \, }\right)-1}{\exp\left(\frac{6\pi}{K_\omega \, }\right)-1/6}$;
\medbreak
\item the region below the graph of $g$  corresponds to flows having an invariant and attracting torus with zero topological entropy  (regular dynamics);
\medbreak
\item the region above the graph of $f$ corresponds to vector fields whose flows exhibit suspended horseshoes (chaotic dynamics).
\end{enumerate}

\bigbreak

The behavior of  $\mathcal{F}_{(A, \lambda)}$ is unknown for the parameter range between the graphs of $f$ and $g$. Nevertheless, for $K_\omega>0$ fixed (along the red dashed line in Figure \ref{graph1}),  the transition from the graph of $g$ to that of $f$ may be explained \emph{via} Arnold tongues associated to the torus bifurcations \cite{AS91, Anishchenko, Aronson}. In particular, in the bifurcation diagram $\left(A, \frac{\lambda}{A}\right)$,  there is a set with positive Lebesgue measure  for which the family \eqref{general2.1} exhibits strange attractors.
\bigbreak

\begin{maintheorem}
\label{thm:F}
 Fix $K_\omega^0>0$. 
In the bifurcation diagram $\left(A, \frac{\lambda}{A}\right)$, where $(A, \lambda)$ is  such that $ g(K_\omega^0) <\frac{\lambda}{A}<f(K_\omega^0)$, there exists a positive measure set  $\Delta$ of parameter values, so that for every $a\in \Delta$, $\mathcal{F}_{(A, \lambda)}$ admits a strange attractor  (of H\'enon-type) with an ergodic SRB measure.
\end{maintheorem}

\bigbreak
 
 The proof of Theorem \ref{thm:F} is a consequence of the \emph{Torus-Breakdown Theory} \cite{AS91, AHL2001, AH2002, Aronson} combined with the results by Mora and Viana \cite{MV93}. A discussion of these results will be performed in Section \ref{torus_bif}.

\subsection*{An application: strange attractors in the unfolding of a Hopf-zero singularity}

By applying the previous theory, we may show the occurrence of strange attractors in a specific case of unfoldings of a Hopf-zero singularity (Type I of \cite{BIS}, Case III of \cite{GH}\footnote{Care is needed to compare both works because the constants choice and signs are  different.}). Since the hypotheses are very technical, we decide to postpone the precise result to Section \ref{Hopf}.

\section{Local and transition maps}\label{localdyn}

In this section we will analyze the dynamics near the Bykov attractor $\Gamma$ through local maps, after selecting appropriate coordinates in neighborhoods of the saddle-foci $O_1$ and $O_2$ (see Figure~\ref{cross_sections}), as done in \cite{OS}.

\begin{figure}[h]
\begin{center}
\includegraphics[height=6.0cm]{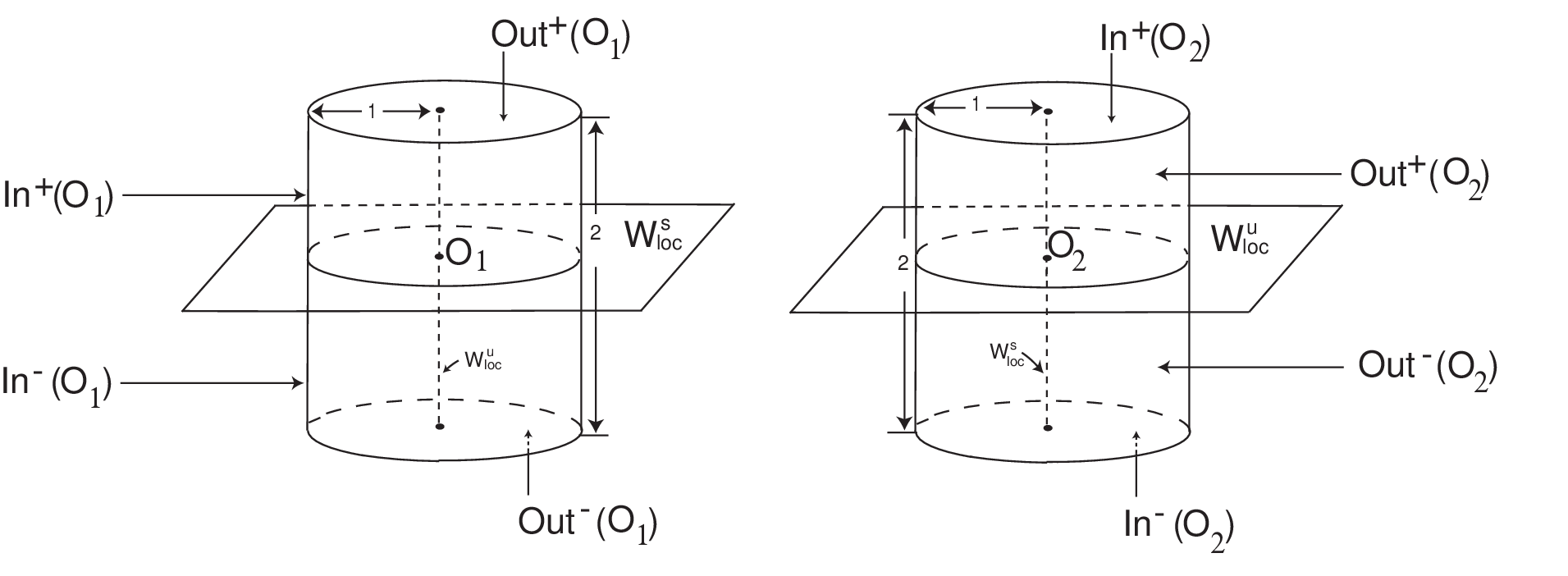}
\end{center}
\caption{Local cylindrical coordinates inside $V_1$ and $V_2$ and near $O_1$ and $O_2$, respectively.}
\label{cross_sections}
\end{figure}

\subsection{Local coordinates}
In order to describe the dynamics around the cycles of $\Gamma$, we use the local coordinates near the equilibria $O_1$ and $O_2$ introduced in  \cite{LR2016}. See also  Ovsyannikov and Shilnikov \cite{OS}.

In these coordinates, we  consider cylindrical neighborhoods  $V_1$ and $V_2$  in ${\RR}^3$ of $O_1 $ and $O_2$, respectively, of radius $\rho=\tilde\varepsilon>0$ and height $z=2\tilde\varepsilon$ --- see Figure \ref{cross_sections}.
After a linear rescaling of the variables, we  may also assume that  $\tilde\varepsilon=1$.
Their boundaries consist of three components: the cylinder wall parametrised by $x\in \RR\pmod{2\pi}$ and $|y|\leq 1$ with the usual cover $$ (x,y)\mapsto (1 ,x,y)=(\rho ,\theta ,z)$$ and two discs, the top and bottom of the cylinder. We take polar coverings of these disks $$(r,\varphi )\mapsto (r,\varphi , \pm 1)=(\rho ,\theta ,z)$$
where $0\leq r\leq 1$ and $\varphi \in \RR\pmod{2\pi}$.
The local stable manifold of $O_1$, $W^s_\loc(O_1)$, corresponds to the circle parametrised by $ y=0$. In $V_1$ we use the following terminology suggested in Figure~\ref{cross_sections}:
\begin{itemize}
\item
$\In(O_1)$, the cylinder wall of $V_1$,  consisting of points that go inside $V_1$ in positive time;
\item
$\Out(O_1)$, the top and bottom of $V_1$,  consisting of points that go outside $V_1$ in positive time.
\end{itemize}
We denote by $\In^+(O_1)$ the upper part of the cylinder, parametrised by $(x,y)$, $y\in\, ]\, 0,1]$ and by $\In^-(O_1)$ its lower part.

The cross-sections obtained for the linearisation around $O_2$ are dual to these. The set $W^s_\loc (O_2)$ is the $z$-axis intersecting the top and bottom of the cylinder $V_2$ at the origin of its coordinates. The set 
$W^u_\loc (O_2)$ is parametrised by $y=0$, and we use:

\begin{itemize}
\item
$\In(O_2)$, the top and bottom of $V_2$,  consisting of points that go inside $V_2$ in positive time;
\item
$\Out(O_2)$,  the cylinder wall  of $V_2$,  consisting of points that go inside $V_2$ in negative time, with $\Out^+(O_2)$ denoting its upper part, parametrised by $(x,y)$, $y\in \, ] 0,1]$ and $\Out^-(O_2)$  its lower part.
\end{itemize}

We will denote by $W^u_{\loc}(O_2)$ the portion of $W^u(O_2)$  that goes from $O_2$ up to $\In(O_1)$ not intersecting the interior of $V_1$ and by $W^s_{\loc}(O_1)$  the portion of $W^s(O_1)$ outside $V_2$ that goes directly  from $\Out(O_2)$ into $O_1$. The flow is transverse to these cross-sections and the boundaries of $V_1$ and of $V_2$ may be written as the closure of  $\In(O_1) \cup \Out (O_1)$ and  $\In(O_2) \cup \Out (O_2)$, respectively. The orientation of the angular coordinate near $O_2$ is chosen to be compatible with the direction induced by the angular coordinate in $O_1$.

\subsection{Local maps near the saddle-foci}
Following \cite{Deng1, OS}, the trajectory of  a point $(x,y)$ with $y>0$ in $\In^+(O_1)$, leaves $V_1$ at
 $\Out(O_1)$ at
\begin{equation}
\Phi_{1 }(x,y)=\left(y^{\delta_1} + S_1(x,y; A, \lambda),-\frac{\omega_1 \, \ln y}{ E_1}+x+S_2(x,y; A, \lambda) \right)=(r,\phi)
\quad \mbox{where}\quad 
\delta_1=\frac{C_{1 }}{E_{1}} > 1,
\label{local_v}
\end{equation}
where $S_1$, $S_2$ are smooth functions which depend on $A$, $\lambda$ and satisfy:
\begin{equation}
\label{diff_res}
\left| \frac{\partial^{k+l+m}}{\partial x^k \partial x^l  \partial \lambda ^m } S_i(x, y;A, \lambda)
\right| \leq \, C\, \,  y^{\delta_1 + \sigma - l}.
\end{equation}
The numbers $C$, $\sigma$ are positive constants and $k, l, m$ are non-negative integers. Similarly, a point $(r,\phi)$ in $\In(O_2) \backslash W^s_{\loc}(O_2)$, leaves $V_2$ at $\Out(O_2)$ at
\begin{equation}
\Phi_{2 }(r,\varphi )=\left(-\frac{\omega_2\, \ln r}{E_2}+\varphi+ R_1(r,\varphi ; A, \lambda),r^{\delta_2 }+R_2(r,\varphi; A, \lambda )\right)=(x,y)
\quad \mbox{where}\quad 
\delta_2=\frac{C_{2 }}{E_{2}} >1 
\label{local_w}
\end{equation}
and $R_1$, $R_2$ satisfy a  condition similar  to (\ref{diff_res}). The terms $S_1$, $S_2$,  $R_1$, $R_2$ correspond to asymptotically small terms that vanish when $y$ and $r$ go to zero. A better estimate under a stronger eigenvalue condition has been obtained in \cite[Prop. 2.4]{Homb2002}.  
\bigbreak

\subsection{The transitions}\label{transitions}
The coordinates on $V_1$ and $V_2$ are chosen so that $[O_1\rightarrow O_2]$ connects points with $z>0$ (resp. $z<0$) in $V_1$ to points with $z>0$  (resp. $z<0$) in $V_2$. Points in $\Out(O_1) \setminus W^u_{\loc}(O_1)$ near $W^u(O_1)$ are mapped into $\In(O_2)$ along a flow-box around each of the connections $[O_1\rightarrow O_2]$. Assuming \textbf{(P8)},  the transition
$$\Psi_{1 \rightarrow  2}\colon \quad \Out(O_1) \quad \rightarrow  \quad \In(O_2)$$
does not depend neither on $\lambda$ nor $A$ and is the Identity map, a choice compatible with Hypothesis \textbf{(P4)}. {Using a more general form for $\Psi_{1 \rightarrow  2}$ would complicate the computations without any change in the final results.}
 Denote by $\eta$ the map
$$\eta=\Phi_{2} \circ \Psi_{1 \rightarrow  2} \circ \Phi_{1 }\colon \quad \In(O_1) \quad \rightarrow  \quad \Out(O_2).$$
Omitting the higher order terms which appear in  \eqref{local_v} and \eqref{local_w}, we infer that, in local coordinates, for $y>0$ we have
\begin{equation}\label{eqeta}
\eta(x,y)=\left(x-K_\omega \ln  y \mod\,2\pi, \,\, \, y^{\delta} \right)
\end{equation}
with
\begin{equation}\label{delta e K}
\delta=\delta_1 \delta_2>1 \qquad \text{and} \qquad  K_\omega= \frac{C_1\omega_2+E_2\omega_1}{E_1 E_2} > 0.
\end{equation}

A similar expression is valid for $y<0$, after suitable changes in the sign of $y$. 
Using  \textbf{(P7)} and \textbf{(P8)}, for $A> \lambda\geq 0$, we still have a well defined transition map
$$\Psi_{2 \rightarrow  1}^{(A, \lambda)}:\Out(O_2)\rightarrow  \In(O_1)$$
that depends on the parameters $\lambda$ and $A$, given by (see Figure \ref{transitions}) 
\begin{equation}\label{transition21}
\Psi_{2 \rightarrow  1}^{(A, \lambda)}(x,y)=\left(x, \,y +A+\lambda \sin x\right).
\end{equation}
\bigbreak

\bigbreak

\begin{remark}
\label{sobre_P8}
Our study is based on a model satisfying Hypothesis \textbf{(P8)}. It governs the transition maps along the heteroclinic connections, being necessary to make precise computations in  Sections \ref{proof Th A}, \ref{Prova Th B} and \ref{proof Th E}. 
The part concerning the transition along $[O_2 \rightarrow O_1]$, $\Psi_{2 \rightarrow  1}^{(A, \lambda)}$,
  corresponds to the expected unfolding from the coincidence of the two-dimensional invariant manifolds at $f_{(0,0)}$; this is also suggested in Figure 4 of \cite{Gaspard}. \end{remark}
\begin{remark}
\label{equivalence}
We are now in conditions to explain why, under \textbf{(P8)}, the hypotheses \textbf{(P7a)}  and \textbf{(P7b)} are equivalent. Indeed, as suggested by Figure \ref{transitions},
the sets $W^u_\loc(O_2)\cap \Out(O_2)$ and $W^s_\loc(O_1)\cap \In(O_1)$ are parametrised by $y=0$. 
Assuming \textbf{(P8)}, the set $$ \Psi_{2 \rightarrow  1}^{(A, \lambda)}(W^u_\loc(O_2)\cap \Out(O_2)) \subset \In(O_1)$$ is the graph of $A+\lambda \sin (x)$, $x\in \EU^1$. Since $A, \lambda \geq 0$, the two-dimensional manifolds $W^u_\loc(O_2)$ and $W^s_\loc(O_1)$ do not intersect if and only if
$$
A + \lambda \sin (x)>0, \qquad x\in \EU^1. 
$$
In particular, $A>\lambda\geq 0$ \emph{i.e.} $ \lambda/A \in [\, 0,1[$. 
\end{remark}

\begin{figure}[h]
\begin{center}
\includegraphics[height=3.7cm]{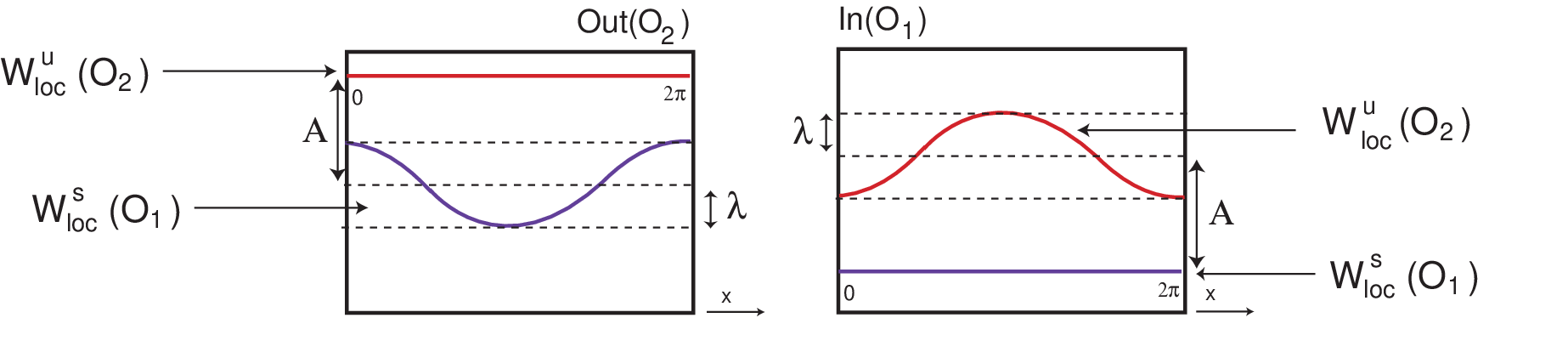}
\end{center}
\caption{\small Geometry of the global map $\Psi_{2 \rightarrow  1}^{(A, \lambda)}$. For $A>\lambda \geq  0$, $W^s_\loc(O_1)$ intersects the wall $\Out(O_2)$ of the cylinder $V_2$ and $W^u_\loc (O_2)$ intersects the wall $\In(O_1)$ of the cylinder $V_1$ on closed curves given, in local coordinates, by the graph of a $2\pi-$periodic function of sinusoidal type. Compare with Figure 4 of \cite{Gaspard}.}
\label{transitions}
\end{figure}

To simplify the notation, in what follows we will sometimes drop the subscript $(A, \lambda)$, unless there is some risk of misunderstanding.  
\section{Proof of Theorem \ref{thm:0}}
\label{proof Th A}
The proof of Theorem \ref{thm:0} is straightforward by composing the local and transition maps constructed in Section \ref{localdyn}.  Let
\begin{equation}\label{first return 1}
\mathcal{F}_{(A, \lambda)} =  \eta \circ \Psi_{2 \rightarrow  1}= \colon \quad  \mathcal{D} \subset \Out(O_2) \quad \rightarrow  \quad \mathcal{D} \subset \Out(O_2).
\end{equation}
be the first return map to $\Out(O_2)$, where $\mathcal{D}$ is the set of initial conditions $(x,y) \in \Out(O_2)$ whose solution returns to $\Out(O_2)$. Composing $\eta$ (\ref{eqeta}) with  $\Psi_{2 \rightarrow  1}$ (\ref{transition21}), the analytic expression of $\mathcal{F}_{(A, \lambda)}$ is given by:
\label{first1}
\begin{eqnarray*}
\mathcal{F}_{(A, \lambda)}(x,y)&=& \left[ -K_\omega \ln  \left[y+A+{\lambda}\sin x\right] +x\,   \pmod{2\pi}, \, \, \left(y + A+{\lambda}\sin x\right)^\delta\right]\\
&=& \left(\mathcal{F}_1^{(A, \lambda)}(x,y), \mathcal{F}_2^{(A, \lambda)}(x,y)\right),
\end{eqnarray*}
and Theorem  \ref{thm:0} is proved. 

\bigbreak
 The following technical result will be useful in the sequel.
\begin{lemma}
\label{contracting1}
For $A>0$ small enough, the  map $\mathcal{F}_2^{(A, \lambda)}$ is a contraction in the variable $y$.
\end{lemma}

\begin{proof} Since $\delta>1$ and $y\in [0,1]$, we may write:
\begin{eqnarray*}
\left|\frac{\partial  \mathcal{F}_2^{(A, \lambda)} (x,y)}{\partial y } \right|&=& \left|\delta (y + A + \lambda\sin x)^{\delta-1} \right|= O(A^{\delta-1})<1,
\end{eqnarray*}
and we get the result.
\end{proof}


\section{Proof of Theorem \ref{thm:B}}
\label{Prova Th B}
To prove the existence of an invariant and attracting smooth closed curve for $\mathcal{F}_{(A, \lambda)}$, we make use if the \emph{Afraimovich Annulus Principle} \cite{AHL2001, AH2002}. We start by proving the existence of a flow-invariant annulus in $\Out^+(O_2)$ in Subsection \ref{ss:annulus}, then we check  the conditions required by the Annulus Principle (one by one) in Subsection \ref{one_by_one}.

\subsection{The existence of an annulus}
\label{ss:annulus}
Let 
$$
\mathcal{B}=\left\{(x,y):  A^\delta \left(1-\frac{\lambda}{A}\right)^\delta  \leq y\leq 2 A^\delta \left(1+\frac{\lambda}{A}\right)^\delta \quad \text{and} \quad x\in \RR \pmod{2\pi}\right\} \subset \Out^+(O_2)
$$
In what follows, if $\mathcal{X}\subset \Out(O_2)$, let $ \stackrel{\ \circ}{\mathcal{X}}$ denote the topological interior of $\mathcal{X}$.
\begin{lemma}
\label{annulus1}
There exists $A_0 \in \, ]\, 0, \varepsilon\, [$ such that: if $A \in\, \,  ] \, 0,A_0\, ]$ and $\lambda \in\,  [\, 0, A\, [$ then $\mathcal{F}_{(A, \lambda)} (\mathcal{B})\subseteq\  \stackrel{\ \circ}{\mathcal{B}}$.
\end{lemma}

\begin{proof}
If $(x,y)\in \mathcal{B}$, we have:

\begin{eqnarray*}
(y+A+\lambda\sin x)^\delta &=& A^\delta \left(\frac{y}{A} +1 + \frac{\lambda}{A} \sin x\right)^\delta \\
 &\leq & A^\delta \left( 2A^{\delta-1}\left( 1+\frac{\lambda}{A}\right)^\delta + 1+ \frac{\lambda}{A}\right)^\delta +o(1)\, y^\delta\\
  &\leq & A^\delta \left( \left(1+\frac{\lambda}{A}\right) \left(2A^{\delta-1}\left( 1+\frac{\lambda}{A}\right)^{\delta-1} + 1\right)\right)^\delta  \\
  &\leq & A^\delta \left( 1+\frac{\lambda}{A}   \right)^\delta \left(O\left(A^\delta\right)+1 \right) \\
  &< & 2 A^\delta \left( 1+\frac{\lambda}{A}   \right)^\delta,  \\
  \end{eqnarray*}
  In particular, there exists $A_1>0$ such that:
  $$
 \forall A<A_1, \, \, \forall \lambda<A,  \qquad (y+A+\lambda\sin x)^\delta<2 A^\delta \left( 1+\frac{\lambda}{A}   \right)^\delta.
  $$
Analogously, if $(x,y)\in \mathcal{B}$, we may write:
\begin{eqnarray*}
(y+A+\lambda\sin x)^\delta &=& A^\delta \left(\frac{y}{A} +1 + \frac{\lambda}{A} \sin x\right)^\delta \\
 &\geq & A^\delta \left( A^{\delta-1}\left( 1-\frac{\lambda}{A}\right)^\delta + 1- \frac{\lambda}{A}\right)^\delta \\
  &\geq & A^\delta \left( \left(1-\frac{\lambda}{A}\right) \left(A^{\delta-1}\left( 1-\frac{\lambda}{A}\right)^{\delta-1} + 1\right)\right)^\delta  \\
  &\geq & A^\delta \left( 1-\frac{\lambda}{A}   \right)^\delta \left(O\left(A^\delta\right)+1 \right)^\delta \\
  &> & A^\delta \left( 1-\frac{\lambda}{A}   \right)^\delta  \\
\end{eqnarray*}
Therefore,  there exists $A_2>0$ such that:
$$
 \forall A<A_2,\, \, \forall \lambda <A \qquad (y+A+\lambda\sin x)^\delta> A^\delta \left( 1-\frac{\lambda}{A}   \right)^\delta.
  $$
Therefore, the region  $\mathcal{B}$ is forward invariant for all $A \in \,\,  ]\, 0, A_0]$ and $\lambda \in [\,0, A\, [$, where $A_0 =~ \min\{A_1, A_2\}$. 
\end{proof}

We recall the annulus principle (version \cite{AH2002}), adapted to our purposes. Let us introduce the following notation: for a vector-valued or matrix-valued function $\textbf{F}(x,y)$, define $$\|\textbf{F} \|=\sup_{(x,y)\in \,  \mathcal{B} }\|\textbf{F}(x,y)\|,$$
where $\|\star\|$ is the standard Euclidian norm.
\begin{theorem}[\cite{AH2002}]
\label{Annulus Principle}
Let  $G(x,y)=\left(x+g_1(x,y),g_2(x,y) \right)$ be a $C^1$  map, with $x\in\RR\pmod{2\pi}$, defined on  the annulus
$$\mathcal{B}=\left\{ (x,y):\ a \leq y \leq b\  \text{and} \  x\in\RR\pmod{2\pi}\right\}, $$ where $0<a<b$,
and satisfying:
\begin{enumerate}
\item $g_1$ and $g_2$ are $2\pi$-periodic in $x$ and  smooth;
\item $G  (\mathcal{B})\subseteq\  \stackrel{\ \circ}{\mathcal{B}}$;
\item $0<\dpt\left\| 1+ \frac{\partial  g_1}{\partial  x}\right\|<1$;
\item $\dpt \left\|\frac{\partial  g_2}{\partial  y}\right\|<1$  (ie, the map $g_2$ is a contraction in $y$);
\item $\dpt 2 \sqrt{\left\|1+\frac{\partial g_1}{\partial x}\right\|^{-2} \left\| \frac{\partial  g_2}{\partial  x} \right\| \left\| \frac{\partial  g_1}{\partial  y} \right\| } < 1- \left\|1+\frac{\partial g_1}{\partial x}\right\|^{-1} \left\|\frac{\partial  g_2}{\partial  y}\right\|$;
\item $\dpt \left\|1+\frac{\partial  g_1}{\partial  x}\right\|+\left\|\frac{\partial  g_2}{\partial  y}\right\|<2$,
\end{enumerate}then the maximal attractor in $\mathcal{B}$ is an invariant closed curve, 
the graph of a $2\pi$-periodic, $C^1$ function $y=h(x)$.
\end{theorem}
\begin{figure}[h]
\begin{center}
\includegraphics[height=5cm]{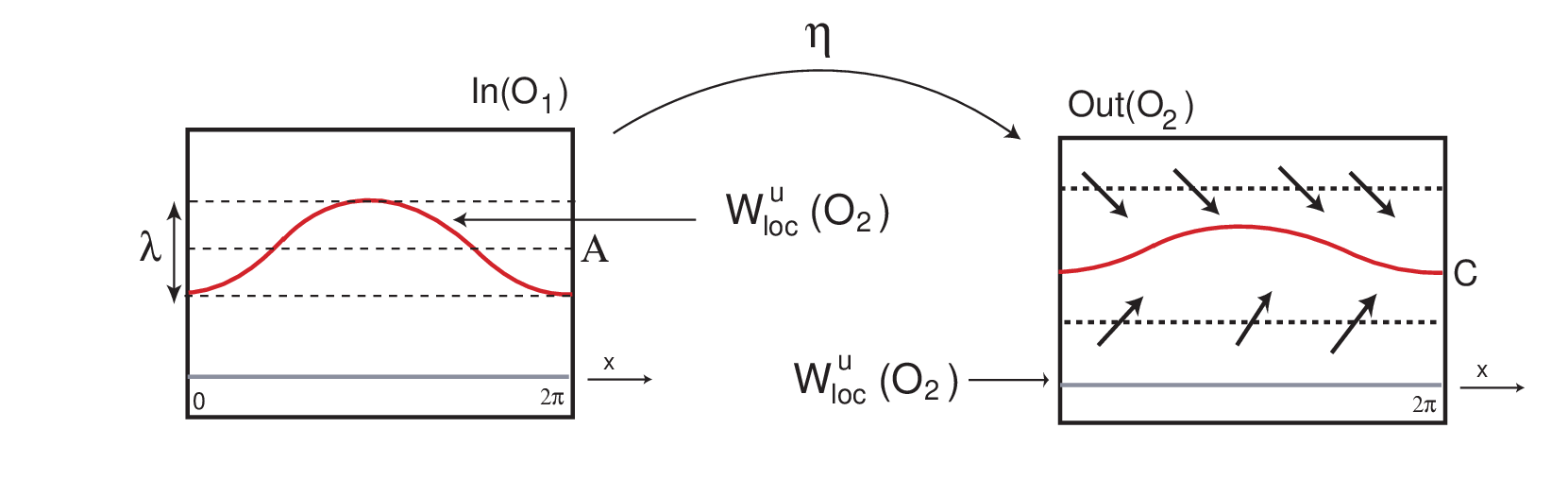}
\end{center}
\caption{\small Illustration of Theorem \ref{thm:B}: for $A>0$ small enough, if $\lambda<\frac{A}{\sqrt{1+K_\omega^2}}$, then there is an invariant closed curve $\mathcal{C}\subset \Out^+(O_2)$ as the maximal attractor of $\mathcal{F}_{(A, \lambda)}$.}
\label{invariant_torus}
\end{figure}

\subsection{Proof of Theorem \ref{thm:B}}
\label{one_by_one}
In this section, we prove Theorem \ref{thm:B} as an application of Theorem~\ref{Annulus Principle}.
Let us define:
\begin{eqnarray*}
g_1(x,y)= \mathcal{F}_1^{(A, \lambda)}(x,y)-x&=& -K_\omega \ln  (y+A + \lambda \sin x) \\ 
 g_2(x,y)=\mathcal{F}_2^{(A, \lambda)}(x,y)&=&(y+A + \lambda \sin x)^\delta
\end{eqnarray*}
where
$$
K_\omega= \frac{C_1\omega_2+E_2\omega_1}{E_1E_2},
$$
whose derivatives may be written as:
\begin{eqnarray*}
\frac{\partial  g_1(x,y)}{\partial x }&=& - K_\omega \frac{\lambda \cos x}{y+ A + \lambda \sin x},\\
\frac{\partial  g_1(x,y)}{\partial  y}&=&\frac{-K_\omega}{y+A+\lambda \sin x}, \\
\frac{\partial  g_2(x,y)}{\partial  x}&=&\delta \lambda  (y+A+\lambda\sin x)^{\delta-1} \cos x, \\
\frac{\partial  g_2(x,y)}{\partial  y}&=& \delta (y + A + \lambda\sin x)^{\delta-1}.
\end{eqnarray*}
We state and prove an auxiliary result that will be used in the sequel.
\begin{lemma}
\label{aux1}
If  $0<\lambda<\frac{A}{\sqrt{1+K_\omega^2}}$, the  range of $\frac{\partial \mathcal{F}_1^{(A, \lambda)}}{\partial x}$ is $]\,0,1\, [$.
\end{lemma}

\begin{proof}
First note that  
\begin{equation}
\label{exp1}
\frac{\partial \mathcal{F}_1^{(A, \lambda)}}{\partial x}= 1+ \frac{\partial g_1(x,y)}{\partial x}= 1- K_\omega \frac{\frac{\lambda}{A}\cos x}{1+\frac{\lambda}{A}\sin x}+o(1)\, y.
\end{equation}

 Taking $a=\frac{\lambda}{A}$, we get:
\begin{eqnarray*}
\frac{\partial^2 \mathcal{F}_1^{(A, \lambda)}}{\partial^2 x}= \frac{\partial \left(1+ \frac{\partial g_1(x,y)}{\partial x}\right)}{\partial x}&=& \frac{\partial }{\partial x}\left({1- K_\omega \, \frac{a \cos x}{ 1 +a \sin x}} \right) \\ \\
&=& K_\omega \frac{a\sin x(1+a\sin(x)) + a^2\cos^2(x)}{(1+a\sin(x))^2}\\ \\
&=& K_\omega \frac{a\sin x +a^2}{(1+a\sin(x))^2}\\
\end{eqnarray*}

Therefore $\frac{\partial^2 \mathcal{F}_1^{(A, \lambda)}(x,y)}{\partial^2 x}=0$ if and only if $a=0$ or $\sin (x)=-a$. Let us define $x^\star\in [\pi, 3\pi/2]$ such that $\sin (x^\star)=-a$.
Therefore:
\begin{eqnarray*}
0<1+ \frac{\partial g_1(x^\star,y)}{\partial x}<1&\Leftrightarrow&0<1- K_\omega \frac{a\cos x^\star}{1+a\sin x^\star}<1\\
&\Leftrightarrow&0<1-\frac{K_\omega a \sqrt{1-a^2}}{1-a^2}<1 \\
&\Leftrightarrow&0<\frac{K_\omega a }{\sqrt{1-a^2}}<1 \\
&\Leftrightarrow&\frac{1-a^2 }{{a^2}}>{K_\omega^2} \\
&\Leftrightarrow&a<\frac{1 }{\sqrt{K_\omega^2+1}} \\
&\Leftrightarrow&\lambda<\frac{A }{\sqrt{K_\omega^2+1}}  \\
\end{eqnarray*}

\end{proof}

In order to prove Theorem \ref{thm:B}, we check one by one the hypotheses of Theorem \ref{Annulus Principle}, putting all pieces together.
\medbreak
\begin{enumerate}
\item It is easy to see to see that $g_1(x,y)= \mathcal{F}_1^{(A, \lambda)}(x,y)-x$ and $g_2(x,y)= \mathcal{F}_2^{(A, \lambda)}(x,y)$ are $2\pi$-periodic in the variable $x$ and smooth in $\mathcal{B}$.
\medbreak
\item This item follows from Lemma \ref{annulus1}, where $a= A^\delta \left(1-\frac{\lambda}{A}\right)^\delta$ and $b= 2 A^\delta \left(1+\frac{\lambda}{A}\right)^\delta$.
\medbreak
\item 
Using Lemma \ref{aux1}, we know that  $0<\left\|1+ \frac{\partial g_1}{\partial x}\right\|<1$ if  $a= \frac{\lambda}{A}\in \left[0, \frac{1}{\sqrt{1+K_\omega^2}}\right[$. In particular, under the same condition, we get $\left\|1+ \frac{\partial g_1}{\partial x}\right \|^{-1}<\infty$.

\medbreak
\item The proof of this item follows from Lemma \ref{contracting1}. Indeed,
\begin{eqnarray*}
\left\|\frac{\partial  g_2(x,y)}{\partial y } \right\|&=& \sup_{(x,y)\in \mathcal{B}} \, |\delta (y + A + \lambda\sin x)^{\delta-1} |= O(A^{\delta-1})<1.
\end{eqnarray*}
\medbreak
\item  Since
\begin{eqnarray*}
\frac{\partial  g_1(x,y)}{\partial y }&=& \frac{-K_\omega}{y+A+\lambda \sin x} = O\left(\frac{1}{A}\right) \qquad \text{and}\\ 
\frac{\partial  g_2(x,y)}{\partial  x}&=&\delta \lambda  (y+A+\lambda\sin x)^{\delta-1} \cos x  \\
&\leq &\delta A  (y+A+\lambda\sin x)^{\delta-1}  \\
&=& O\left(A^{\delta}\right),
\end{eqnarray*}
this implies that: 
\begin{eqnarray*}
\sqrt{\left\| 1+\frac{\partial g_1}{\partial x}\right\|^{-2}  \left\|\frac{\partial g_2}{\partial x}\right\| \left\|\frac{\partial g_1}{\partial y}\right\| } &=& \left\|1+ \frac{\partial g_1}{\partial x}\right\|^{-1} O(A^{\delta/2}) \, O(A^{-1/2})=O\left(A^{\frac{\delta-1}{2}}\right).
\end{eqnarray*}
Combining now
\begin{eqnarray*}
1- \left\| 1+\frac{\partial  g_1}{\partial  x}\right\|^{-1}\left\|\frac{\partial  g_2}{\partial  y}\right\| &=& 1- O\left(A^{\delta-1}\right)
\end{eqnarray*}
with
$$
 \sqrt{\left\| 1+\frac{\partial g_1}{\partial x}\right\|^{-2}  \left\|\frac{\partial g_2}{\partial x}\right\| \left\|\frac{\partial g_1}{\partial y}\right\| } =O\left(A^{\frac{\delta-1}{2}}\right), 
$$
item (5) follows for $A>0$ sufficiently small.
\medbreak

\item Using again Lemmas \ref{contracting1} and \ref{aux1}, we get:
$$
\left\| 1+\frac{\partial  g_1}{\partial  x}\right\| + \left\| \frac{\partial  g_2}{\partial  y}\right\| <1+ O(A^{\delta-1})<2
$$

\end{enumerate}

This ends the proof of the existence of an attracting curve for the dynamics of  $ \mathcal{F}_{(A, \lambda)}$, provided   $A>0$ is sufficiently small  and $\frac{\lambda}{A}<\frac{1}{\sqrt{1+K_\omega^2}}$. This curve, shown is Figure \ref{invariant_torus},  is not contractible because it is the graph of a $C^1$-smooth map $h$.

\bigbreak

\subsection*{Geometric interpretation}
Condition $\lambda<\frac{A}{\sqrt{1+K_\omega^2}}$ implies that the image of the annulus $\mathcal{B}$ under $\mathcal{F}_{(A, \lambda)}$  is also an annulus bounded by two curves without folds. The subsequent image of this annulus is self-alike too, and so on. As a result, we obtain a sequence of embedded annuli;  the contraction in the $y$-variable (Lemma \ref{contracting1}) guarantees that these annuli intersect in a single and smooth attracting closed curve.

\section{Proof of Theorem \ref{thm:E}}
\label{proof Th E}
In this section, under an appropriate hypothesis,  we prove the existence of two rectangles whose image under $\mathcal{F}_{(A, \lambda)}$ overlaps with $\Out(O_2)$ at least twice. Therefore, we obtain a construction similar to the hyperbolic Smale horseshoe: a small tubular neighborhood of $y=0$ (in $\Out(O_2)$) is folded and mapped into itself.

\subsection{Stretching the angular component}

The first technical result, illustrated in Figure \ref{stretching1},  says that the image under $\mathcal{F}_1^{(A, \lambda)}$ of the  segment parametrised by
$$
\left\{(x,y) \in \mathcal{D}: \quad \frac{\pi}{2}<x< \frac{3\pi}{2} \qquad \text{and}\qquad y=y_0\right\}
$$
($y_0 \in [0, 1]$)  is monotonic, meaning that there are no folds. Here we use the hypothesis that the map $\Phi(x)=\sin x$ has non-degenerate critical points.

\begin{lemma}
\label{increasing_lemma}
For any $y_0 \in [0, 1]$, the angular map $\mathcal{F}_1^{(A, \lambda)}(x,y_0)$ is an increasing map for $x\in\, \,  ]\, \pi/2, 3\pi/2 \, [$.
\end{lemma}

\begin{proof}
Let $y_0 \in [0, 1]$. One knows that:
\begin{eqnarray*}
\frac{\partial \mathcal{F}_1^{(A, \lambda)}}{\partial x} (x, y_0)&=&1-K_\omega \frac{\lambda \cos x}{y_0+A+\lambda\sin x} \\
&=&1-K_\omega \frac{\frac{\lambda}{A} \cos x}{\frac{y_0}{A}+1+\frac{\lambda}{A}\sin x} \\
&\approx &1-K_\omega \frac{\frac{\lambda}{A} \cos x}{1+\frac{\lambda}{A}\sin x} \\
\end{eqnarray*}
Since $\frac{\lambda}{A}\ll1$ and $\frac{\lambda}{A} \cos x<0$ in $]\, \pi/2, 3\pi/2\,[$, we may conclude that:
$$
\forall x \in\, \,  ]\, \pi/2, 3\pi/2 \, [, \qquad 
\frac{\partial \mathcal{F}_1^{(A, \lambda)}}{\partial x}  (x, y_0)>0
$$
and the result follows.
\end{proof}

From now on, for $n\in \NN\backslash\{1\}$, let $\theta_n=\pi +\arcsin (c_n)$ where $c_n=\frac{1}{2(n+1)}$.
 It is easy to see that $\theta_n \in \, \,  ]\, \pi , 3\pi/2\, [$ and $\sin (\theta_n)=\sin (\pi+\arcsin (c_n))=-c_n<0.$ 
Therefore, for all $y\in [0,1]$ and $\delta \in [\, 0, 3\pi/2 -\theta_n\, [\, \, \subset \, \,   ]\, 0, \pi/2\, [$, we have:
\begin{eqnarray}
\label{serie_complicada}
&&\mathcal{F}_1^{(A, \lambda)}(3\pi/2-\delta, y) - \mathcal{F}_1^{(A, \lambda)}( \theta_n, y) =  \nonumber \\  \nonumber  \\&=& -K_\omega \ln  [y+A + \lambda \sin(3\pi/2-\delta)]+(3\pi/2 -\delta) + \nonumber \\ 
&&K_\omega  \ln  (y+A + \lambda \sin \theta_n) - \theta_n  \nonumber  \\  \nonumber  \\
&=& (3\pi/2 -\delta -\theta_n) +K_\omega \ln  \left[\frac{y+A+\lambda\sin \theta_n}{y+A+\lambda \sin(3\pi/2 -\delta)}\right]  \nonumber  \\ \nonumber  \\
&=& (3\pi/2- \delta -\theta_n) +K_\omega \ln  \left[\frac{\frac{1}{A}y+1+\frac{\lambda}{A}\sin \theta_n}{\frac{1}{A}y+1+\frac{\lambda}{A} \sin(3\pi/2 -\delta)}\right]  \nonumber  \\  \nonumber \\
&=& (3\pi/2- \delta -\theta_n) +K_\omega \ln  \left[\frac{1-\frac{\lambda \, c_n}{A}}{1-\frac{\lambda}{A} \cos(\delta)}\right] +o(1)\, y  
\end{eqnarray}

\begin{figure}[h]
\begin{center}
\includegraphics[height=7cm]{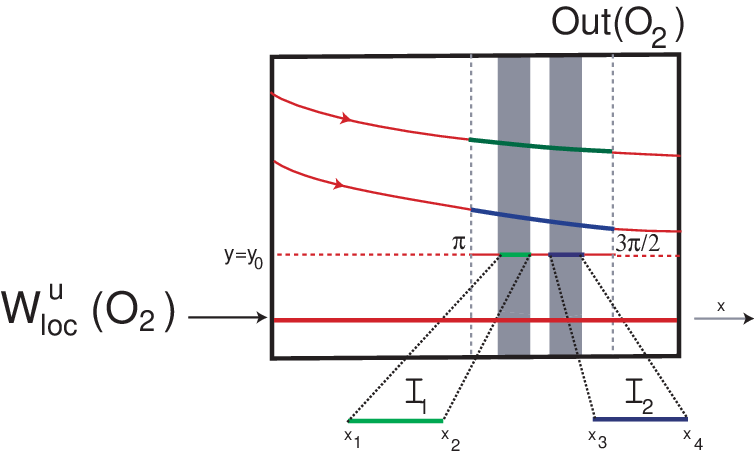}
\end{center}
\caption{\small  For $y_0\in [0,1]$, the image under $\mathcal{F}_{(A, \lambda)}$ of the segment $[\pi, 3\pi/2] \times \{y_0\}$ is a curve (without folds) intersecting twice the rectangle $R= [\pi, 3\pi/2] \times [-1,1] \subset \Out(O_2).$ There are two sub-segments ($I_1$ and $I_2$) which are stretched by $\mathcal{F}_{(A, \lambda)}$ along~$R$.}
\label{stretching1}
\end{figure}

For each $n\in \NN\backslash\{1\}$, $\delta \in  \, ]\,0, \pi/2\, [$, $a=\frac{\lambda}{A}\in \, [0,1[$ and $A\neq 0$, define the map
\begin{equation}
\label{Def_P}
P\left(\delta,a \right):= K_\omega \ln \left[\frac{1-a\, c_n}{1-a \cos(\delta)}\right].
\end{equation}
\begin{lemma}
\label{two turns}
For $n\in \NN\backslash\{1\}$, $\frac{\lambda}{A} >  \frac{\exp\left(\frac{\pi}{K_\omega \, c_n}\right)-1}{\exp\left(\frac{\pi}{K_\omega \, c_n}\right)-c_n}$ if and only if $P\left(0,  \frac{\lambda}{A} \right)>\frac{\pi}{c_n}$. 
\end{lemma}

\begin{proof} Taking $a= \lambda / A$, we may write: 
\begin{eqnarray*}
P(0, a) > \frac{\pi}{c_n} &\Leftrightarrow& K_\omega \ln  \left(\frac{1-a\, c_n}{1-a\, \cos 0}\right) > \frac{\pi}{c_n}\\
&\Leftrightarrow&  \ln  \left(\frac{1-a\, c_n}{1-a}\right) > \frac{\pi}{ K_\omega \, c_n}\\
&\Leftrightarrow&   \frac{1-a\, c_n}{1-a}> \exp\left(\frac{\pi}{ K_\omega \, c_n}\right)\\
&\Leftrightarrow&   {1-a\, c_n}> (1-a)\exp\left(\frac{\pi}{ K_\omega \, c_n}\right)\\
&\Leftrightarrow&   {1-a\, c_n}> \exp\left(\frac{\pi}{ K_\omega \, c_n}\right)- a \exp\left(\frac{\pi}{ K_\omega \, c_n}\right)\\
&\Leftrightarrow&  a\left( \exp\left(\frac{\pi}{ K_\omega \, c_n}\right)-c_n\right)> \exp\left(\frac{\pi}{ K_\omega \, c_n}\right)-1\\
&\Leftrightarrow&  a> \frac{\exp\left(\frac{\pi}{ K_\omega \, c_n}\right)-1}{\left( \exp\left(\frac{\pi}{ K_\omega \, c_n}\right)-c_n\right)}\\
\end{eqnarray*}
\end{proof}


Using Lemma \ref{two turns}, by continuity of $P$ with respect to $\delta$, there exists $\delta_0<3\pi/2 -\theta_n$ such that 
\begin{equation}
\label{equivalence1}
\frac{\lambda}{A}> \frac{\exp\left(\frac{\pi}{ K_\omega \, c_n}\right)-1}{\left( \exp\left(\frac{\pi}{ K_\omega \, c_n}\right)-c_n\right)} \qquad \Rightarrow \qquad P\left(\delta_0,\frac{\lambda}{A}\right)>\frac{\pi}{c_n}= 2(n+1)\pi. 
\end{equation}
Taking into account \eqref{serie_complicada}, the condition $3\pi/2- \delta -\theta_n \in [0, \pi/2],$ 
and \eqref{equivalence1}, we conclude that there exists $\delta_0<3\pi/2 -\theta_n$ such that
$$
\mathcal{F}_1^{(A, \lambda)}(3\pi/2-\delta_0, y) - \mathcal{F}_1^{(A, \lambda)}( \theta_n, y)> 2 n \pi + \frac{3\pi}{2}.
$$
Then,
as suggested in Figure \ref{stretching1},  the image,  under $\mathcal{F}_{(A, \lambda)}$, of the segment $[\pi, 3\pi/2] \times \{y_0\}$ is a curve (without folds) intersecting $n$ times the rectangle $$R= [\pi, 3\pi/2] \times [-1,1] \subset \Out(O_2).$$  In particular,  we may find $n$  disjoint intervals defined by $I_n=[x_{2n-1}, x_{2n}]$  such that
\begin{equation}
\label{ordem2}
\pi<\theta_n<x_1<x_2<\ldots<x_{2n-1}<x_{2n}<\frac{3\pi}{2}-\delta_0,
\end{equation}
for which  we define $n$ non-empty compact and disjoint subsets $H_n$, on which the map $\mathcal{F}_{(A, \lambda)}|_{H_1\cup\ldots\cup H_n}$ is topologically conjugate to a Bernoulli shift with $n$ symbols.  The construction of this horseshoe ($n=2$)  is the goal of the Subsection \ref{horseshoe_8.2}.

\subsection{The construction of the topological horseshoe}
\label{horseshoe_8.2}
We now recall the main steps of the construction of the horseshoes (with $n=2$) which are  $\mathcal{F}_{(A, \lambda)} $-invariant. The argument uses the generalized Conley-Moser conditions \cite{GH, Koon, Wiggins} to ensure the existence for $\frac{\lambda}{A} >  \frac{\exp\left(\frac{\pi}{K_\omega \, c_2}\right)-1}{\exp\left(\frac{\pi}{K_\omega \, c_2}\right)-c_2}$, of an invariant set  $\Lambda\subset \Out^+(O_2)$ topologically conjugated to a Bernoulli shift with 2 symbols.   As suggested in Figure \ref{horseshoe2} (right), in this subsection, we decided to flip coordinates $(x,y) \leftrightarrow (y,x)$ in $\mathcal{D}\subset \Out(O_2)$ because doing so, we get a high similarity of the present situation to that of \cite{Rodrigues3, Wiggins}, where we address the reader for  details.

\bigbreak

Given a rectangular region $\mathcal{R}$ in $\Out(O_2)$, parameterised by a rectangle $[w_1,w_2]\times [z_1, z_2]$, a \emph{horizontal strip} in $\mathcal{R}$ is a set
$$\mathcal{H}=\{(y,x): x\in[u_1(y),u_2(y)]\qquad y \in \,[w_1,w_2]\}$$
where $u_1,u_2: [w_1,w_2] \rightarrow [z_1,z_2]$ are Lipschitz functions such that $u_1(y)<u_2(y)$. The \emph{horizontal boundaries} of a horizontal strip are the graphs of the maps $u_i$; the \emph{vertical boundaries} are the lines $\{w_i\} \times  [u_1(w_i),u_2(w_i)]$. 
In an analogous way, we define a \emph{vertical strip across} $\Out(O_2)$, a \emph{vertical rectangle}, with the roles of $x$ and $y$ reversed.

\medbreak

\begin{figure}[h]
\begin{center}
\includegraphics[height=7.4cm]{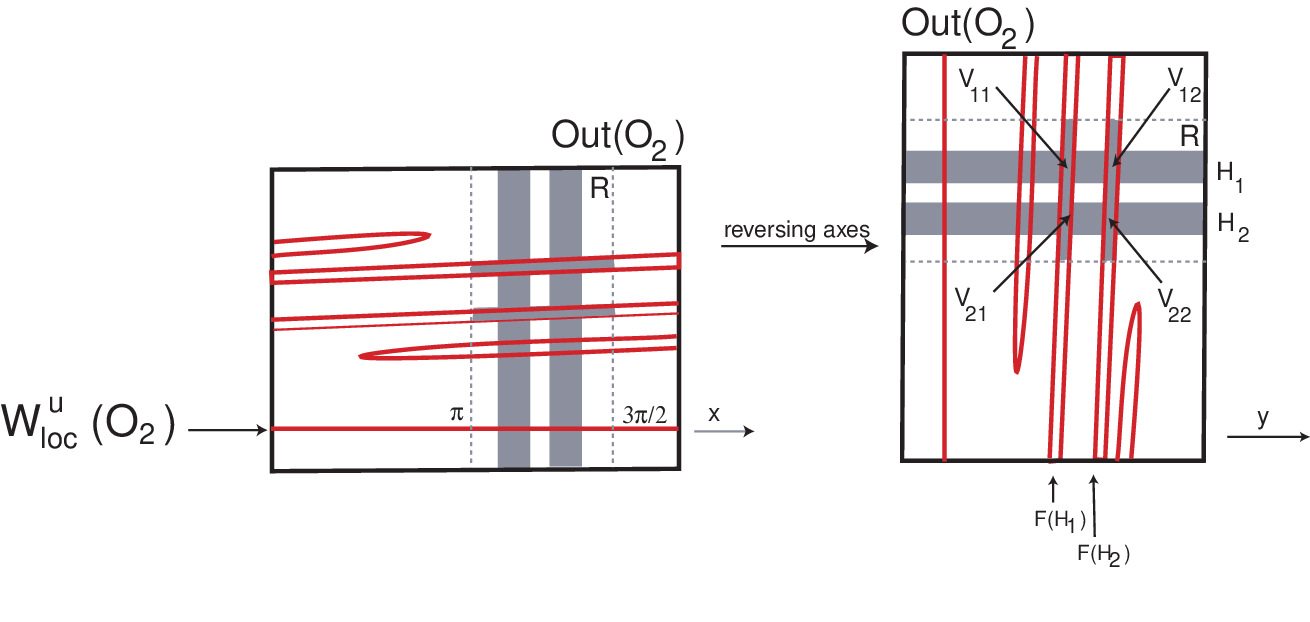}
\end{center}
\caption{\small  Two rectangles whose image under $\mathcal{F}_{(A, \lambda)}$ overlaps with $\Out(O_2)$ at least twice, giving rise to a Smale horseshoe -- see Figure \ref{horseshoe_turns}(D).}
\label{horseshoe2}
\end{figure}

We may deduce (by construction) that:

\medbreak

\begin{enumerate}
\item For $i=1,2$, as suggested in Figure \ref{horseshoe2},  the horizontal strip 
$$
H_i= \{(y,x)\in \mathcal{D}: x\in I_i\}.
$$
 is mapped (homeomorphically) by $\mathcal{F}_{(A, \lambda)}$ into a vertical strip across $ R\subset \Out(O_2)$.

\medbreak
\item $\mathcal{F}_{(A, \lambda)}(H_1)\cap\mathcal{F}_{(A, \lambda)}(H_2)=\emptyset$ because $\mathcal{F}_{(A, \lambda)}$ is a diffeomorphism, where it is well defined.

\medbreak
\item  $\mathcal{F}_{(A, \lambda)}(H_1)$ and $\mathcal{F}_{(A, \lambda)}(H_2)$ have full intersections with $H_1$ and $H_2$ (\emph{i.e.} the vertical boundaries of $\mathcal{F}_{(A, \lambda)}(H_1)$ and $\mathcal{F}_{(A, \lambda)}(H_2)$ cross both the horizontal and vertical boundaries of $H_1$ and $H_2$) because of Lemma \ref{two turns} and subsequent remark;

\medbreak
\item For $i, j=1,2$, defining $V_{ji}: =\mathcal{F}_{(A, \lambda)}(H_i) \cap H_j$, $H_{ij}: = \mathcal{F}_{(A, \lambda)}^{-1}(V_{ji}) =H_i \cap \mathcal{F}_{(A, \lambda)}^{-1}(H_j)$ and denoting  by $\partial_v V_{ji}$ the vertical boundaries of $V_{ji}$, we get, by construction, that:

\begin{enumerate}
\medbreak
\item $\partial_v V_{ji} \subset \partial_v \mathcal{F}_{(A, \lambda)}(H_i)$;
\medbreak
\item the map $\mathcal{F}_{(A, \lambda)}$ maps $H_{ij}$ homomorphically onto $V_{ji}$;
\medbreak
\item $\mathcal{F}_{(A, \lambda)}^{-1}(\partial_v V_{ji}) \subset \partial_v H_i$.
\end{enumerate}
\end{enumerate}

Using \cite{Wiggins}, we may conclude that there exists a $\mathcal{F}_{(A, \lambda)}-$invariant set of initial conditions $${\Lambda}=\bigcap_{n \in \ZZ} \mathcal{F}_{(A, \lambda)}^n(H_1 \cup H_2)$$ on which the map $\mathcal{F}_{(A, \lambda)}|_\Lambda$ is topologically conjugate to a Bernoulli shift with two symbols.

\subsection{Hyperbolicity}

To prove the hyperbolicity of $\Lambda$ with respect to the map $\mathcal{F}_{(A, \lambda)}$, we apply the following result due to Afraimovich, Bykov and Shilnikov \cite{ABS}.

\begin{theorem}[\cite{ABS}]
\label{thm:hyp}
Let $H:U \rightarrow \RR^2$ be a $C^1$ map where $U$ is an open convex subset of $\RR^2$ such that $H(x,y):=(F_1(x,y), F_2(x,y))$, where $x, y \in \RR$. If:
\begin{enumerate}
\item $\left\|\frac{\partial F_2}{\partial y}\right\| <1$
\item $\left\|\left(\frac{\partial F_1}{\partial x}\right)^{-1}\right\|<1$;
\item $1-{\left\|\frac{\partial F_2}{\partial y}\right\|}{\left\|\left(\frac{\partial F_1}{\partial x}\right)^{-1}\right\|}>2\sqrt{{\left\|\frac{\partial F_2}{\partial x}\right\|} \left\|\frac{\partial F_1}{\partial y}\right\|\left\|\left(\frac{\partial F_1}{\partial x}\right)^{-1}\right\|} $
\item $\left(1-\left\|\frac{\partial F_2}{\partial y}\right\| \right)\left(1-\left\|\frac{\partial F_1}{\partial x}\right\|^{-1}\right)>{\left\|\frac{\partial F_2}{\partial x}\right\|\left\|(\frac{\partial F_1}{\partial x})^{-1}\right\|} \left\|\frac{\partial F_1}{\partial y}\right\|$,
\end{enumerate}
then any compact invariant set $\Lambda\subset U$ is hyperbolic.
\end{theorem}

To finish the proof of Theorem \ref{thm:E}, we check one by one the hypotheses of Theorem \ref{thm:hyp}, when applied to the compact set $\Lambda\subset \Out^+(O_2)$, where $F_1= \mathcal{F}_1^{(A, \lambda)}$ and $F_2= \mathcal{F}_2^{(A, \lambda)}$.
\medbreak
\begin{enumerate}
\item The proof follows from Lemma \ref{contracting1}.
\medbreak
\item One knows that:
\begin{eqnarray*}
\left\|\left(\frac{\partial F_1}{\partial x}\right)\right\|&=& \sup_{(x, y)\in \Lambda} \left|\left(\frac{\partial F_1}{\partial x}\right)\right| \\
&=& \sup_{(x, y)\in \Lambda } \left|1-K_\omega \frac{\frac{\lambda}{A} \cos x}{1+\frac{\lambda}{A}\sin x} + o(1)\, y \right|>1
\end{eqnarray*}
The last inequality follows from the fact that  $\pi<x<\frac{3\pi}{2}$ (see \eqref{ordem2}) in $\Lambda \subset H_1\cup H_2$. 

\medbreak
\item By Lemma \ref{contracting1}, it follows that:
$$
1-{\left\|\frac{\partial F_2}{\partial y}\right\|}{\left\|\left(\frac{\partial F_1}{\partial x}\right)^{-1}\right\|}= 1- O(A^{\delta-1}){\left\|\left(\frac{\partial F_1}{\partial x}\right)^{-1}\right\|}
$$
On the other hand, we may write:
$$
2\sqrt{{\left\|\frac{\partial F_2}{\partial x}\right\|} \left\|\frac{\partial F_1}{\partial y}\right\|\left\|\left(\frac{\partial F_1}{\partial x}\right)^{-1}\right\|} = 2 \sqrt{\left\|\left(\frac{\partial F_1}{\partial x}\right)^{-1}\right\|}\,\,O(A^{\frac{\delta-1}{2}})
$$

\medbreak
\item Similarly, we write:
$$
\left(1-\left\|\frac{\partial F_2}{\partial y}\right\| \right)\left(1-\left\|\frac{\partial F_1}{\partial x}\right\|^{-1}\right) = (1-O(A^\delta))\left(1-\left\|\frac{\partial F_1}{\partial x}\right\|^{-1}\right)
$$
and 
$$
{\left\|\frac{\partial F_2}{\partial x}\right\|\left\|\left(\frac{\partial F_1}{\partial x}\right)^{-1}\right\|} \left\|\frac{\partial F_1}{\partial y}\right\| = O(A^\delta) \left\|\left(\frac{\partial F_1}{\partial x}\right)^{-1}\right\|\, O(A^{-1}) $$
\medbreak
\end{enumerate}

\subsection*{Geometric interpretation}
Condition $a=\frac{\lambda}{A} >  \frac{\exp\left(\frac{\pi}{K_\omega \, c_2}\right)-1}{\exp\left(\frac{\pi}{K_\omega \, c_2}\right)-c_2}$ provides enough expansion in the $x$-variable (angular coordinate) within the region $R\subset \Out(O_2)$. It means that there are at least two rectangles whose image under $\mathcal{F}_{(A, \lambda)}$ overlaps with $R\subset \Out(O_2)$ at least twice, giving rise to a suspended horseshoe. Hyperbolicity is obtained by considering rectangles whose vertical boundaries do not have reversals of orientation.

\begin{figure}[h]
\begin{center}
\includegraphics[height=11cm]{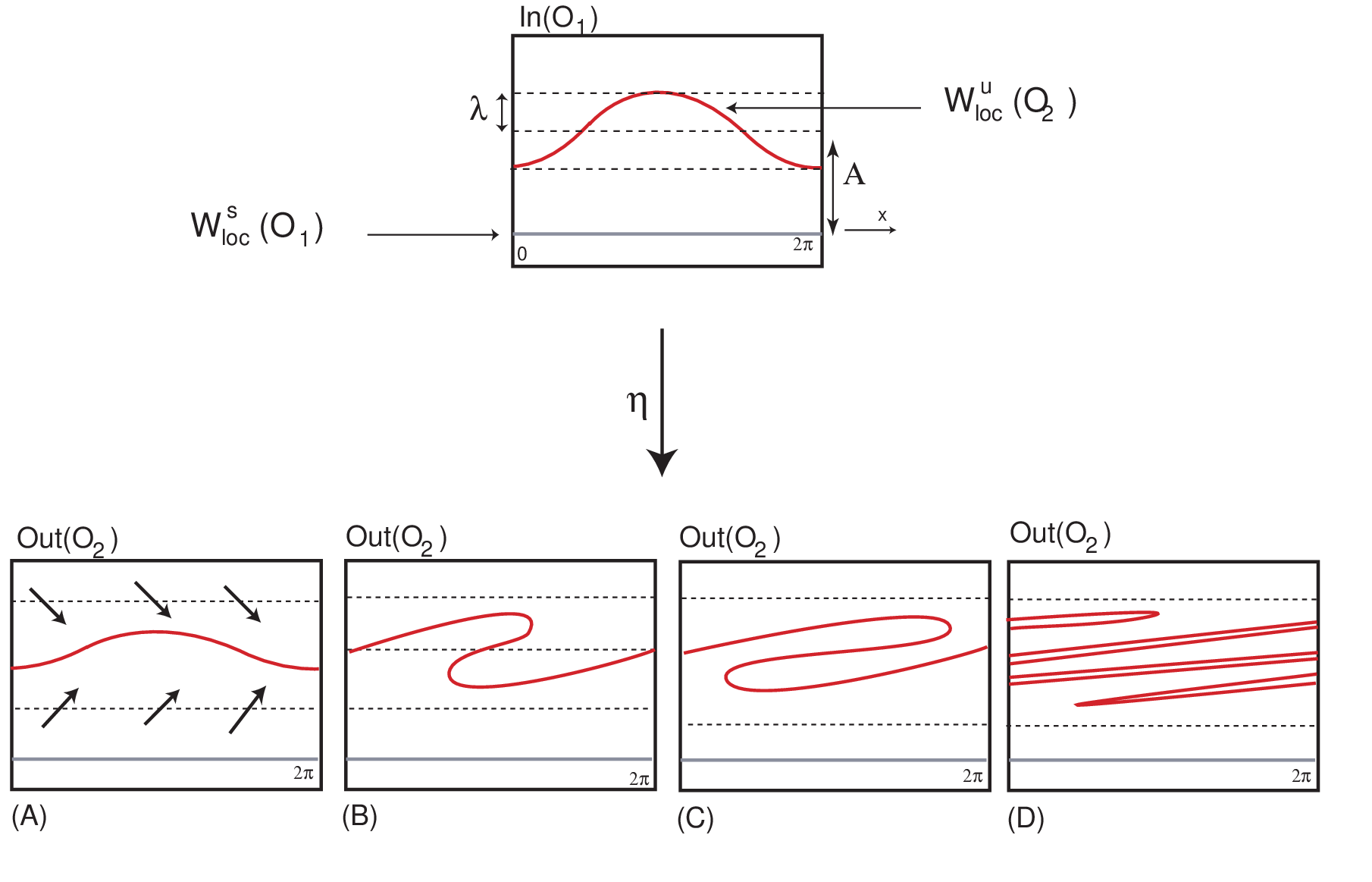}
\end{center}
\caption{\small Image of $\eta(W^u(O_2)\cap \In(O_1))$ for different values of $\frac{\lambda}{A}$ with $K_\omega^0$ fixed. Transition from an invariant and attracting curve (A) to a horseshoe (D) for a fixed $K_\omega^0>0$ and $\lambda/A$ increasing. One observes  the ``\emph{breaking of the wave}'' which accompanies the break of the invariant circle $\mathcal{C}$. In (D), a neighborhood of $y=0$ is folded and mapped into itself, leading to the formation of horseshoes -- see Figure \ref{horseshoe2} .}
\label{horseshoe_turns}
\end{figure}

\begin{remark}
The dynamics of $\Lambda$ is mainly governed by the geometric configuration of the global invariant manifold $W^u(O_2)$. 
\end{remark}

\section{Proof of Theorem \ref{thm:F}}
\label{torus_bif}

In the bifurcation diagram $\left(\frac{\lambda}{A}, K_\omega \right)$, we may draw  two smooth curves in the first quadrant, the graphs of $g$ and $f$,  such that:
\begin{enumerate}
\medbreak
\item  $g(K_\omega)=\frac{1}{\sqrt{1+K_\omega^2}}$ and $f(K_\omega)=  \frac{\exp\left(\frac{6\pi}{K_\omega \, }\right)-1}{\exp\left(\frac{6\pi}{K_\omega \, }\right)-1/6}$;
\medbreak
\item the region below the graph of $g$  corresponds to flows having an invariant and attracting torus with zero topological entropy  (regular dynamics);
\medbreak
\item the region above the graph of $f$ corresponds to vector fields whose flows exhibit chaos (chaotic dynamics).
\end{enumerate}
For $A, K_\omega>0$ fixed, as $\lambda$ increases, one observes the ``\emph{breaking of the wave}'' which accompanies the break of the invariant circle $\mathcal{C}$. As suggested in Figure \ref{horseshoe_turns}, the attracting curve $\mathcal{C}$ (whose existence is ensured by Theorem \ref{thm:B}) starts to disintegrate into a finite collection of periodic saddles and sinks, a phenomenon occurring within an Arnold tongue. Once the horseshoes develop, they persist. 

\begin{figure}[h]
\begin{center}
\includegraphics[height=8.9cm]{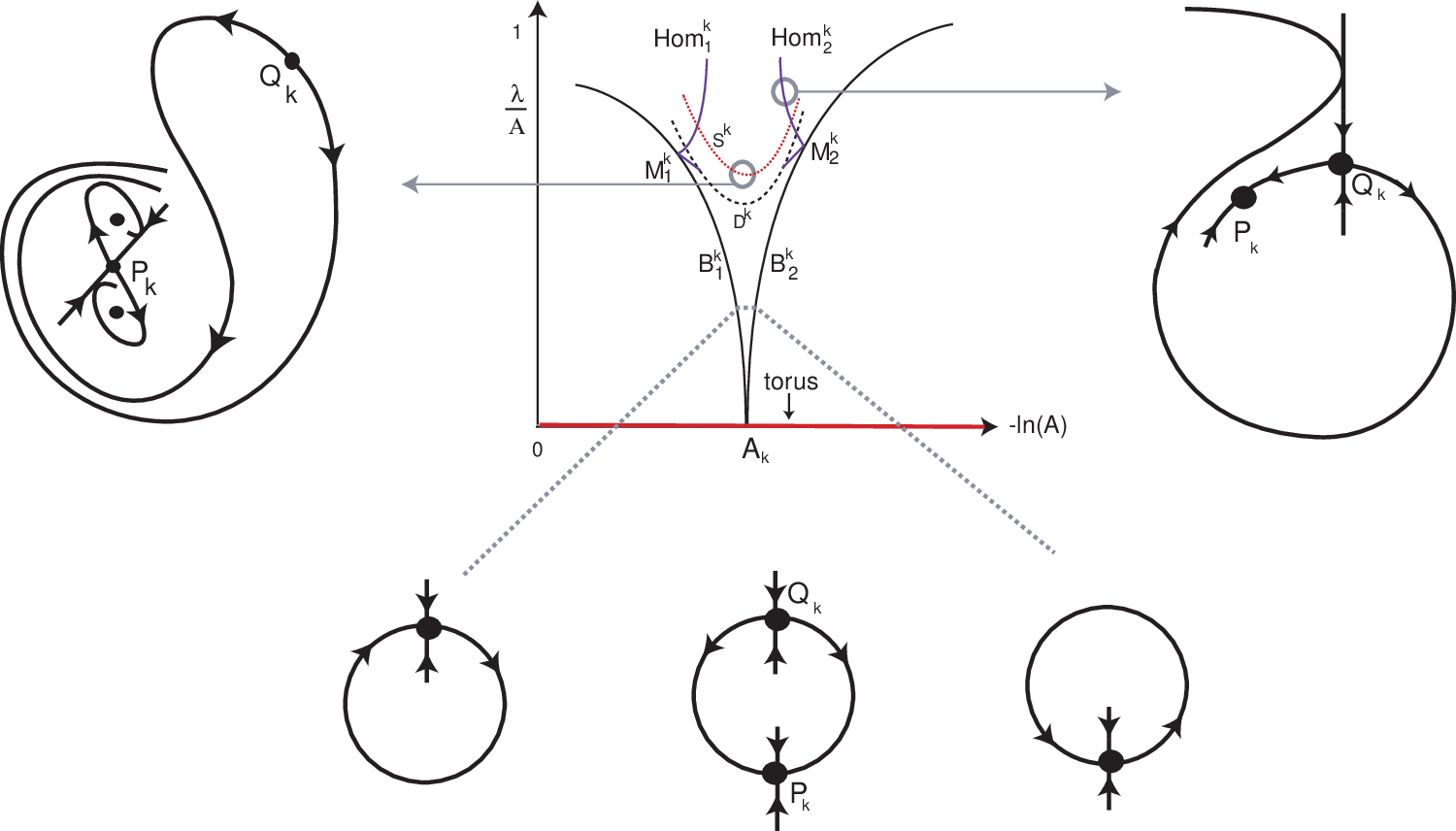}
\end{center}
\caption{\small In the Arnold tongue $\mathcal{T}_k$, there are three lines corresponding to the emergence of a homoclinic tangency associated to a dissipative periodic point. (Left): homoclinic tangency of third class which occurs along $S^k$; (Right): homoclinic tangencies which occur along the lines $\text{Hom}_1^k$, $\text{Hom}_2^k$. (Center): $B_k^1$ and $B_k^2$: saddle-node bifurcations; $D^k$: period doubling bifurcation (torus-breakdown); $M_k^1, M_k^2$: above these points, the maximal invariant set is not homeomorphic to a circle; $Q_k$: saddle; $P_k$: sink.  Source: \cite{Aronson, Shilnikov_tutorial}.}
\label{homoclinic_torus}
\end{figure}
\bigbreak
In what follows, we describe generic mechanisms to break an attracting two-dimensional torus which involves the onset of homoclinic tangencies produced by the stable and unstable manifolds of a dissipative saddle\footnote{Note that, for small $A, \lambda>0$, the first return map $\mathcal{F}_{(A, \lambda)}$ is still contracting.}. These tangencies are the source of strange attractors. We address the reader to \cite{AS91, AHL2001, AH2002, Aronson} for more information on the subject. A comprehensive description of these phenomena has been given by Shilnikov \emph{et al} \cite{Shilnikov_tutorial}.

\subsection{Dissecting an Arnold tongue}
\label{Arnold tongue ss}
For $\varepsilon>0$ small, the choice of parameters in Section \ref{s:setting} ($\varepsilon>A>\lambda \geq 0$) lets us to build  the bifurcation diagram of Figure \ref{homoclinic_torus}, in the  domain $$\{0 \leq -\ln  A , \, \, 0 <  \lambda/A <1\} \quad \subset \quad [0, +\infty[ \, \, \times\, \, [0,1].$$
We suggest that the reader follows this subsection observing Figure \ref{homoclinic_torus}.
\medbreak
 
 Within the region above, for eack $k\in \NN$, we may define an Arnold tongue (or resonance wedge), denoted by $\mathcal{T}_k$, adjoining the horizontal axis at a point $A_k=(\exp(-2\pi k), \, 0).$  Parameters within this wedge correspond to Poincar\'e maps with at least a pair of fixed points; one of the fixed points is always of saddle-type (say $Q_k$); the other point is a sink (say $P_k$).  We suppose just one pair of fixed points for the following analysis (as shown in Figure \ref{homoclinic_torus}). 
 \medbreak
  The borders of the Arnold tongue $\mathcal{T}_k$ are the bifurcation curves $B_1^k$ and $B_2^k$ on which the fixed points merge to a saddle-node.   The curve $B_2^k$ continue up to the line $\lambda/A = 1$, while the curve $B_1^k$ bends to the left staying below $\lambda/A= 1$. Eventually these curves may touch the corresponding curves of other tongue, meaning that there are parameter values for which  the periodic points of periods $ k$ and $ m$ coexist, $m, k\in \NN$.

 The points $M_1^k$ and $M_2^k$ correspond to \emph{pre-wiggles}: below these points, in $B_1^k$ and $B_2^k$,  the limit set of $W^u(Q_k)$ is the saddle-node itself and  is homeomorphic to a circle. Above this point, the maximal invariant set is not homeomorphic to a circle.  There is also a curve, say $D^k$, above which the invariant torus (or the curve $\mathcal{C}$) no longer exists due to period doubling bifurcation process \cite{Anishchenko}.  
 \medbreak
 In the next subsection we will emphasise the role played by the lines $\text{Hom}_1^k$, $\text{Hom}_2^k$ and $S^k$, also depicted in Figure \ref{homoclinic_torus}.

\medbreak


\subsection{Strange attractors}

Continuing the process of dissecting an Arnold tongue, the authors of \cite{AS91, Anishchenko} describe generic mechanisms  by which the invariant and attracting torus is destroyed. Two of them are revived in the next result and involve  homoclinic tangencies -- routes [PA] and [PB] of \cite{Anishchenko}. 

\begin{theorem}[\cite{AS91, Anishchenko}, adapted]
\label{three_lines}
For $K_\omega^0>0$ fixed, in the bifurcation diagram $\left(-\ln A, \frac{\lambda}{A}\right)$, within $\mathcal{T}_k$:
\begin{itemize}  
\item[(1)]   there are two curves $\text{Hom}_1^k$ and $\text{Hom}_2^k$ corresponding to a homoclinic tangency associated to a dissipative periodic point of the first return map $\mathcal{F}_{(A, \lambda)}$.
\item[(2)] there is one curve $S^k$ corresponding to a homoclinic tangency (of third class) associated to a dissipative periodic point of the first return map $\mathcal{F}_{(A, \lambda)}$.
\end{itemize}
 \end{theorem}

Along the bifurcation curves $\text{Hom}_1^k$ and $\text{Hom}_2^k$, one observes a homoclinic contact of  the components $W^s(Q_k)$ and $W^u(Q_k)$, where $Q_k$ is a dissipative saddle, as illustrated in Figure \ref{homoclinic_torus}(right).  
The curves $\text{Hom}_1^k$ and $\text{Hom}_2^k$ divide the region above  $D^k$ into two regions with simple and complex dynamics. In the zone above the curves $\text{Hom}_1^k$ and $\text{Hom}_2^k$, there is a fixed point $Q_k$ exhibiting a transverse homoclinic intersection, and thus the corresponding  map $\mathcal{F}_{(A, \lambda)}$ exhibits nontrivial hyperbolic chaotic sets. Other stable points of large period exist in the region above the curves $\text{Hom}_1^k$ and $\text{Hom}_2^k$ since the homoclinic tangencies arising in these lines are generic \cite{Colli98, Newhouse79}. The curve $S^k$ corresponds to a homoclinic tangency of third class meaning that there are tangencies associated to the fixed points which have emerged from the period-doubling bifurcation at $D^k$ (see the meaning of $D^k$ in Subsection \ref{Arnold tongue ss}). Again, above these curves, in the parameter space $(-\ln A, \frac{\lambda}{A})$, there is a dense set of parameters for which the map $\mathcal{F}_{(A, \lambda)}$ has infinitely many sinks.
\medbreak
The next result will be used to finish the proof of  Theorem \ref{thm:F}:

\begin{theorem}[\cite{MV93}]
\label{MV_th}
Let $(f_\mu)_\mu$ a one-parameter family of diffeomorphisms on a surface $S$ and suppose that for $f_{\mu_0}$ has a homoclinic tangency associated to a dissipative periodic point $q\in S$. Then, under generic conditions, there is a positive Lebesgue measure set $E$ of parameter values near $\mu_0$ such that for all $\mu\in E$, the diffeomorphism $f_\mu$ exhibits a H\'enon-like strange attractor near the orbit of tangency (with an ergodic SRB measure). 

\end{theorem}

By Theorem \ref{three_lines}, the existence of $\text{Hom}_1^k$, $\text{Hom}_2^k$ and $S^k$ shows that there are curves in the space of parameters   $\left(A, \frac{\lambda}{A} \right)$ for which the corresponding first return map has a quadratic (generic) homoclinic tangency associated to a dissipative periodic point of the first return map $\mathcal{F}_{(A, \lambda)}$. Using now Theorem \ref{MV_th}, there exists a positive measure set  $\Delta$ of parameter values, so that for every $a\in \Delta$, $\mathcal{F}_{(A, \lambda)}$ admits a strange attractor  of H\'enon-type with an ergodic SRB measure. This completes the proof of Theorem \ref{thm:F}.

\begin{remark}
In this type of result, the number of connected components with which the strange attractors intersect the section $\Out(O_2)$ is not specified nor is the size of their basins of attraction.
\end{remark}

\begin{remark}
\label{just one cycle}
In \textbf{(P4)} we have asked for the existence of two 1D-connections. Nevertheless the statements of Theorems \ref{thm:B},  \ref{thm:E} and  \ref{thm:F} still hold if \textbf{(P4)} is replaced by:
\medbreak
\begin{itemize}
\item[\textbf{(P4a)}]\label{B4a} There is one trajectory contained in  $W^u(O_1)\cap W^s(O_2)$,
\end{itemize}
\medbreak
\noindent
provided the map $\Psi_{2\rightarrow 1}^{(A, \lambda)}$ sends the line $W^u_\loc (O_2)\, \cap \, \Out(O_2)$ into the connected component of $\In(O_1)\backslash W^s_\loc(O_1)$ where solutions follow the 1D-connection. This remark will be important in Section \ref{Hopf}.
\end{remark}

\section{An application: Hopf-zero singularity unfolds strange attractors}

\label{Hopf}
In this section, we prove the existence of strange attractors in particular analytic unfoldings of a Hopf-zero singularity. In order to improve the readability of the paper, we recall to the reader the most important steps about unfoldings of a Hopf-zero singularity. 
\medbreak

From now on,  we consider Hopf-Zero singularities, that is, three-dimensional vector fields $f^*$ in $\RR^3$ such that:
\begin{itemize}
\medbreak
\item $O\equiv (0,0,0)$ is an equilibrium of $f^\star$;
\medbreak
\item the spectrum of $df^*(0, 0, 0)$ is $\{\pm i\omega, \, 0\}$, with $\omega>0$.
\medbreak
\end{itemize}
Without loss of generality, we can assume that:
\begin{equation}
\label{nf1}
Df^*(0,0,0)= 
\left[ {\begin{array}{ccc}
   0& \omega & 0 \\
   -\omega & 0 & 0 \\
   0&0&0 \\
  \end{array} } \right].
  \end{equation}

Observe that the lowest codimension singularities in $\RR^3$ with a three-dimensional center manifold are the ones whose linear part is linearly conjugated to \eqref{nf1}.
\subsection{The normal form}
The normal form of a degenerate jet with linear part given by
$
\left[ {\begin{array}{ccc}
   0& \omega & 0 \\
   -\omega & 0 & 0 \\
   0&0&0 \\
  \end{array} } \right]
  \left[ {\begin{array}{c}
  x \\
  y \\
  z \\
  \end{array} } \right]
  $
may be written in cylindrical coordinates $(r, \theta, z)$ by:
\begin{equation}
\label{family_unfolding2}
\left\{ 
\begin{array}{l}
\dot{r}=a_1 rz +a_2 r^3  +a_3 rz^2+ O(\left|r,z\right|^4)\\  \\
\dot{\theta}=\omega+O(\left|r,z\right|^2) \\ \\
\dot{z}=b_1 r^2 +b_2 z^2 +b_3 r^2 z + b_4 z^3+O(\left|r,z\right|^4)
\end{array}
\right.
\end{equation}
where $\omega>0$ and $a_1, a_2, a_3, b_1, b_2, b_3, b_4 \in \RR \backslash\{0\}$.
Normal form can be chosen in such a way that, up to arbitrarily high $k\in \NN$, the truncation of order $k$ contains no $\theta$-dependent terms \cite{GH}. Ultimately, one has to restore the tail which, generically, contains no $\theta$-dependent terms\footnote{The author is grateful to one of the reviewers for pointing out this remark.}.
Truncating \eqref{family_unfolding2} at order 2 and removing the angular coordinate $\theta$, we obtain:

\begin{equation}
\label{family_unfolding2.1}
\left\{ 
\begin{array}{l}
\dot{r}=a_1 rz \\  \\
\dot{z}=b_1r^2 +b_2 z^2.
\end{array}
\right.
\end{equation}

\begin{figure}[h]
\begin{center}
\includegraphics[height=8cm]{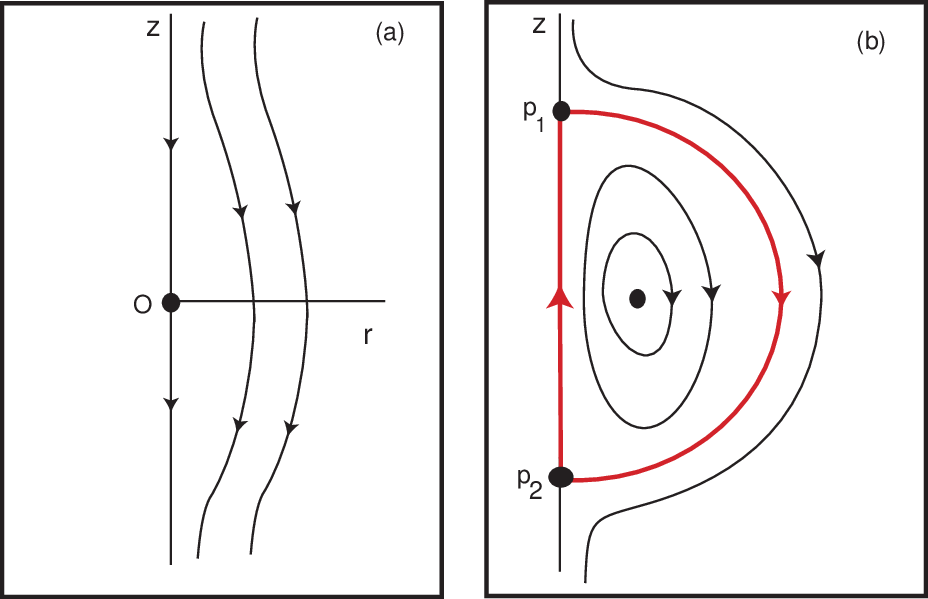}
\end{center}
\caption{\small (a) Phase diagram of \eqref{family_unfolding2.2} for $a>0$ and $b=-1$. (b) Stable heteroclinic cycle associated to $O_1$ and $O_2$ for the differential equation \eqref{family_unfolding2.3} with $\mu_1=0$.}
\label{HZ1b}
\end{figure}

Setting 
$
r^{new} = -\sqrt{|b_1 b_2|}\, r$, $z^{new}=-b_2 z,$ and 
dropping the superscripts ``\emph{new}'', we get the differential equation:

\begin{equation}
\label{family_unfolding2.2}
\left\{ 
\begin{array}{l}
\dot{r}=a rz \\  \\
\dot{z}=br^2-z^2
\end{array}
\right.
\end{equation}
where
$
a=-a_1/b_2 $ and $ b=\pm 1.$  The phase diagram of \eqref{family_unfolding2.2} for $a>0$ and $b=-1$ is shown in Figure \ref{HZ1b}(a).
\bigbreak

Takens \cite{Takens74} proved that there are six topological types for the normal form \eqref{family_unfolding2.2}, but from now on, we are only interested in the one characterized by the conditions $b=-1$ and $a>0$ (Type I of \cite{BIS}; Case III of \cite{GH}).
The line defined by $r=0$ ($z$-axis) is flow-invariant. According to \cite{GH}, any generic unfolding of \eqref{family_unfolding2.2}, truncated at second order, may be written as:

\begin{equation}
\label{family_unfolding2.3}
\left\{ 
\begin{array}{l}
\dot{r}=\mu_1r+a rz \\  \\
\dot{z}=\mu_2-r^2-z^2
\end{array}
\right.
\end{equation}
whose flow satisfy the following properties (for $\mu_2\geq 0$ and ${\mu_2}>\frac{\mu_1^2}{a^2}$):

\begin{itemize}
\medbreak
\item there are two equilibria of saddle-type, say $p_1=(0, \sqrt{\mu_2})$, $p_2=(0, -\sqrt{\mu_2})$, whose eigenvalues of $df^\star$ at the equilibria are $\mu_1 \pm a\sqrt{\mu_2}$ and $\mp 2\sqrt{\mu_2}$;
\item there is another equilibrium given by  $\left(  \sqrt{\mu_2-\frac{\mu_1^2}{a^2}}, -\frac{\mu_1}{a}\right)$  which is a center;
\item for $\mu_1=0$, there is a heteroclinic cycle associated to $p_1$ and $p_2$;
\item for $\mu_1=0$, if $G(r,z)= \frac{a}{2} r^\frac{2}{a} \left(\mu_2 -\frac{r^2}{1+a}-z^2\right) $ then  the Lie derivative of $G$ with respect to the vector field associated to \eqref{family_unfolding2.3} satisfies the inequality:
$$
\mathcal{L}_v \, G \equiv 0,
$$
meaning that there is a family of non-trivial periodic solutions limiting the inner part of the planar heteroclinic cycle associated to $p_1$ and $p_2$. This cycle is Lyapunov stable -- see Figure \ref{HZ1b}(b).
\end{itemize}

The truncated normal form of order 2 is not enough to our purposes because the heteroclinic cycle is not asymptotically stable.

\subsection{Truncating at order 3}
Truncating at third order any generic unfolding of \eqref{family_unfolding2.2} with $a>0$ and $b=-1$, we obtain:
\begin{equation}
\label{family_unfolding2.3bb}
\left\{ 
\begin{array}{l}
\dot{r}=\mu_1 r+a rz + cr^3+drz^2 \\ \\
\dot{z}=\mu_2-r^2-z^2 + er^2z+fz^3.
\end{array}
\right.
\end{equation}
Assuming that the new parameters $c,d,e,f\in \RR$ satisfy the open conditions:
\begin{equation}
\label{Hopf_constants}
3c+e>d+3f \qquad \text{and } \qquad {3(3c+e)+d+3f}<0,
\end{equation}
the flow of \eqref{family_unfolding2.3bb} exhibits a heteroclinic cycle associated to $$\tilde{p}_1 \approx (0,  \sqrt{\mu_2}+ f/2) \qquad \text{and}  \qquad \tilde{p}_2 \approx (0,  -\sqrt{\mu_2}+ f/2)$$  at all points  $(\mu_1, \mu_2)$ lying at the line $AHC$ defined by  $$AHC: \qquad \mu_2=\frac{-4\mu_1}{3(3c+e)+d+3f}+O(\sqrt{\mu_2}).$$ This cycle has a non-empty basin of attraction \cite{GH, Kuznetsov}. Line AHC is depicted in Figure \ref{HZ1a}.

\begin{figure}[h]
\begin{center}
\includegraphics[height=10cm]{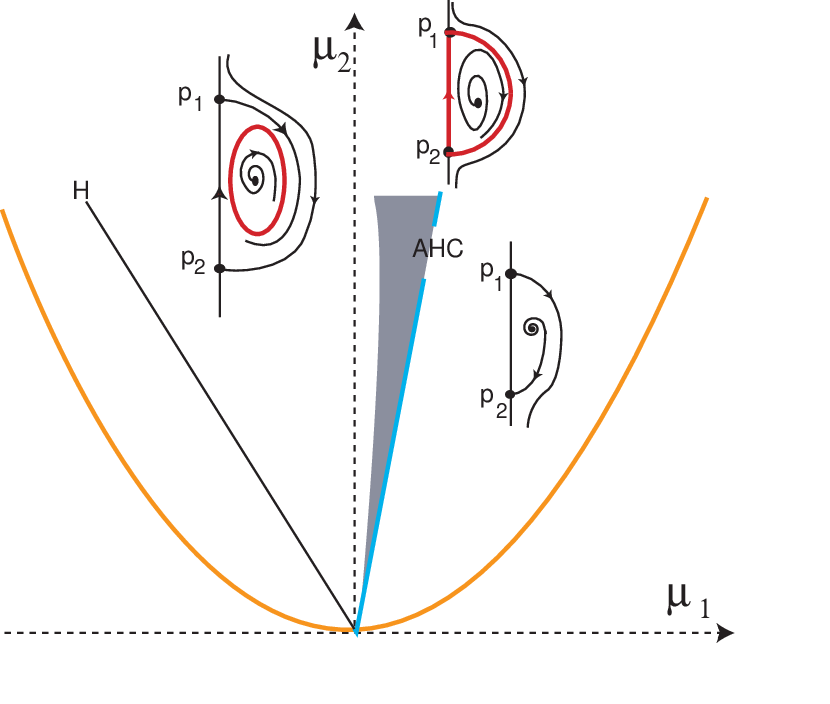}
\end{center}
\caption{\small Bifurcation diagram for the differential equation \eqref{family_unfolding2.3b} when  $3c+e>d+3f$ and ${3(3c+e)+d+3f}<0$. (H): Hopf bifurcation; (AHC): attracting heteroclinic cycle. The grey region is the wedge shaped region $\mathcal{W}$ where Proposition \ref{thm:G} is valid. }
\label{HZ1a}
\end{figure}

\subsection{General perturbations}
Adding the angular coordinate $\theta$ to equation \eqref{family_unfolding2.3bb}, we define a $\mathbb{SO}(2)$-equivariant vector field, say $ f_{(\mu_1, \mu_2)}.$
 Its flow has cycle associated to the lift of $\tilde{p}_1$ and $\tilde{p}_2$, say $O_1$ and $O_2$ with non-empty basin of attraction. This cycle is made by one 1D and one 2D heteroclinic connections associated to two hyperbolic saddles-foci with different Morse indices. The coincidence of the invariant manifolds of the hyperbolic saddle-foci is exceptional and they are expected to split.

 Generic analytic unfoldings of the Hopf-zero singularities were considered in \cite{DIKS}; the authors introduced an extra parameter  $\varepsilon = \sqrt{\mu_2}$ and obtained a singular perturbation problem with a pure rotation when $\varepsilon=0$ or a family with rotation speed tending to $+\infty$ as $\varepsilon \rightarrow 0$. Since the imaginary part of the eigenvalues of the vector field at the equilibria has the form $O(1/\varepsilon)$, this means that $K_\omega \rightarrow 0$; details in Section 3 of  \cite{DIKS}.
In the present  work, we are interested in particular unfoldings of a Hopf-zero singularity. 
 
 \begin{definition}
 A \emph{Gaspard-type unfolding} of  a Hopf-zero singularity $f^\star$ (with $a>0$ and $b=-1$ in \eqref{family_unfolding2.2})  has the form:
  \begin{equation}
\label{family_unfolding2.3b}
\left\{ 
\begin{array}{l}
\dot{r}=\mu_1 r+a rz + cr^3+drz^2 + O(\varepsilon) \mathcal{H}_1(r,  z) \\ \\
\dot{\theta} = 1 + O(\varepsilon)\mathcal{H}_2(r,  z) \\ \\
\dot{z}=\mu_2-r^2-z^2 + er^2z+fz^3 + O(\varepsilon)  \mathcal{H}_3(r,  z)
\end{array}
\right.,
\end{equation}
where $\mathcal{H}_1, \mathcal{H}_2, \mathcal{H}_3 $ are arbitrary smooth functions of degree greater than three in the variables $r$ and $z$  (\emph{i.e.} they have the form of equations (2.15) and (2.16) of Gaspard \cite{Gaspard}).
   \end{definition}
 
\begin{proposition}
\label{thm:G}

Let $f^\star$ be a Hopf-zero singularity (with $a>0$ and $b=-1$ in \eqref{family_unfolding2.2}). 
Then there exists an analytic curve $\mathcal{S}$ in the parameter space  $(\mu_1, \mu_2)$ and a domain $\mathcal{W}$ contained in a wedge shaped neighborhood of $\mathcal{S}$ 
such that: if $(\mu_1, \mu_2)\in  \mathcal{W}$ then any Gaspard-type unfolding of $f^\star$ contains strange attractors  (of H\'enon-type) with an ergodic SRB measure.

\end{proposition}

In the following proof,  Hypothesis \textbf{(P8)} is implicit; it corresponds to the expected unfolding from the coincidence of the two-dimensional invariant manifolds of the equilibria; see Remark~\ref{sobre_P8}.

\begin{proof} We apply Theorem  \ref{thm:F} to prove Proposition  \ref{thm:G}. 
Indeed, if $(\mu_1, \mu_2)\in AHC=:\mathcal{S}$,  then the flow of  $ f_{(\mu_1, \mu_2)}$ satisfies:
\medbreak
\begin{itemize}
\item there are two hyperbolic saddle-foci $O_1$ and $O_2$ satisfying \textbf{(P1)}--\textbf{(P2)};
\medbreak
\item  the manifolds $W^u(O_2)$ and $W^s(O_1)$ coincide and  one branch of  $W^u(O_1)$ coincide with $W^s(O_2)$ -- see Remark \ref{just one cycle}. In particular, there is a heteroclinic cycle associated to $O_1$ and $O_2$ with non-empty basin of attraction, meaning that  \textbf{(P3)}--\textbf{(P4a)} are satisfied;
\medbreak
\item by construction on the way the angular coordinate is acting on \eqref{family_unfolding2.3bb}, the saddle-foci have the same chirality -- \textbf{(P5)} is valid;
\medbreak
\item by hypothesis, we perturb $f_{(\mu_1, \mu_2)}$ in such a way that the manifolds $W^u(O_2)$ and $W^s(O_1)$ do not intersect and the one-dimensional manifolds are preserved, emerging an attracting 2-torus.  The dynamics on it exhibits intervals of frequency locking and irrational flow as the rotation number varies. This perturbation correspond to \textbf{(P6)}--\textbf{(P7)}. 
\end{itemize}
\bigbreak

Adding the $ O(\varepsilon)$-terms   will correspond to generic perturbations (without symmetry) that break the attracting two-dimensional torus \cite{GH}. 
Observe that  $\lim_{\varepsilon \rightarrow 0} K_\omega =+\infty$ since the imaginary part of the eigenvalues of the vector field at the equilibria has the form $O(1/\varepsilon)$ (cf. pp. 4444 of \cite{DIKS}). By Theorem \ref{thm:F}, the result follows, where the parameter $\varepsilon$ plays the role of $\lambda$. 
\end{proof}

\begin{remark}
Besides the break of the two-dimensional attracting torus, adding generic $ O(\varepsilon)$-terms to $ f_{(\mu_1, \mu_2)}$ may imply that the 1D-connection is also broken. This does not affect the proof of Proposition \ref{thm:G} since the existence of strange attractors has been obtained \emph{via Torus-breakdown phenomena}, which occur even when the 1D-connection is broken.
\end{remark} 
 
Theorem  \ref{thm:F} cannot be applied directly for generic unfoldings of the Hopf-zero singularity because there is no chance to express a generic Hopf-zero singularities as perturbations of a vector field whose flow has an attracting heteroclinic cycle.

 \subsection{Discussion and conjecture}
 Following Bonckaert and Fontich \cite{BF2005}, when the scaling parameters defined in \cite{DIKS} tend to $0$, the invariant manifolds have a limit position given by the invariant manifolds of the equilibria at the $2$-jet of the singularity  when $z>0$. Therefore, for any generic unfolding of the Hopf-zero singularity (Type I of \cite{BIS}), the splitting distance is well defined for the $1D$ and $2D$-connections.  The first case was obtained in \cite{BCS2013} --  the distance $\mathcal{S}^1$ between the 1D invariant manifolds is exponentially small with respect to $\varepsilon>0$ and  the coefficient in front of the dominant term depends on the full jet of the singularity. The splitting function for the 2D invariant manifolds, say $\mathcal{S}^2$, has been obtained in \cite{BCS2018}.


When the manifolds $W^u(O_2)$ and $W^s(O_1)$ intersect transversely, conclusive results are given in \cite{BIS}: any generic analytic unfolding of a Hopf-zero singularity exhibits Shilnikov bifurcations and strange attractors \cite{BV, Homb2002}. 

 When the two-dimensional manifolds of the saddle-foci do not intersect (see \cite{BCS2018} and case $\sigma \gg O(\varepsilon)$ of \cite{BIS}\footnote{The condition $ \sigma \gg O(\varepsilon)$ does not refer to a class of unfoldings, but to a region in the parameter space associated to the unfolding.}), there exist trapping regions which prevent the existence of Shilnikov homoclinic cycles.  In particular, the route \emph{Shilnikov bifurcations} $\Rightarrow$ \emph{Strange attractors} is not possible.  
 With that configuration (the invariant manifolds of the saddle-foci do not intersect), the existence of suspended horseshoes may be ensured by our Theorem \ref{thm:E} provided its hypotheses are satisfied --  see Remark  \ref{just one cycle}. Distances $\mathcal{S}^1$ (\cite{BCS2013}) and $\mathcal{S}^2$ (\cite{BCS2018}) have to be introduced in the definition of the transition maps. We conjecture that generic analytic unfoldings of a Hopf-zero singularity also include strange attractors created by the Torus-breakdown mechanism. We defer this task for a future work.

\section*{Acknowledgments}
Special thanks to Santiago Ib\'a\~nez (Univ. Oviedo) and Isabel Labouriau (Univ. Porto) for fruitful discussions. The author is also grateful to the three referees for the constructive comments, corrections  and suggestions which helped to improve the readability of this manuscript.

\appendix
\section{Glossary}
\label{Definitions}

For $\varepsilon>0$ small, consider the two-parameter family of $C^3$-smooth autonomous differential equations
\begin{equation}
\label{general2}
\dot{x}=f_{(A, \lambda)}(x)\qquad x\in \EU^3  \qquad A, \lambda \in [0, \varepsilon] 
\end{equation}
Denote by $\varphi_{(A, \lambda)}(t,x)$, $t \in \RR$, the associated flow.

\subsection{Symmetry}
Given a group $\mathcal{G}$ of endomorphisms of $\EU^3$, we will consider two-parameter families of vector fields $(f_{(A, \lambda)})$ under the equivariance assumption $f_{(A, \lambda)}(\gamma x)=\gamma f_{(A, \lambda)}(x)$ for all $x \in \EU^3$, $\gamma \in \mathcal{G}$ and $(A, \lambda )\in  [0, \varepsilon]^2.$
For an isotropy subgroup $\widetilde{\mathcal{G}}< \mathcal{G}$, we will write $\Fix(\widetilde{\mathcal{G}})$ for the vector subspace of points that are fixed by the elements of $\widetilde{\mathcal{G}}$. Observe that, for $\mathcal{G}-$equivariant differential equations, the subspace $\Fix(\widetilde{\mathcal{G}})$ is flow-invariant.

\subsection{Attracting set}
A subset $\Omega$ of a topological space $\mathcal{M}$ for which there exists a neighborhood $U \subset \mathcal{M}$ satisfying $\varphi(t,U)\subset U$ for all $t\geq 0$ and $\bigcap_{t\,\in\,\RR^+}\,\varphi(t,U)=\Omega$ is called an \emph{attracting set} by the flow $\varphi$, not necessarily connected. Its basin of attraction, denoted by $\textbf{B}(\Omega)$ is the set of points in $\mathcal{M}$ whose orbits have $\omega-$limit in $\Omega$. We say that $\Omega$ is \emph{asymptotically stable} (or that $\Omega$ is a \emph{global attractor}) if $\textbf{B}(\Omega)=\mathcal{M}$. An attracting set is said to be \emph{quasi-stochastic} if it encloses periodic solutions with different Morse indices, structurally unstable cycles, sinks and saddle-type invariant sets.

\subsection{Heteroclinic structures}
Suppose that $O_1$ and $O_2$ are two hyperbolic saddle-foci of $f_{(A, \lambda)}$ with different Morse indices (dimension of the unstable manifold). There is a {\em heteroclinic cycle} associated to $O_1$ and $O_2$ if
$W^{u}(O_1)\cap W^{s}(O_2)\neq \emptyset$ and  $W^{u}(O_2)\cap W^{s}(O_1)\neq \emptyset.$ For $i, j \in \{1,2\}$, the non-empty intersection of $W^{u}(O_i)$ with $W^{s}(O_j)$ is called a \emph{heteroclinic connection} between $O_i$ and $O_j$, and will be denoted by $[O_i \rightarrow  O_j]$. Although heteroclinic cycles involving equilibria are not a generic feature within differential equations, they may be structurally stable within families of systems which are equivariant under the action of a compact Lie group $\mathcal{G}\subset \mathbb{O}(n)$, due to the existence of flow-invariant subspaces \cite{GH}.

\subsection{Bykov cycle}
A heteroclinic cycle between two hyperbolic saddle-foci of different Morse indices, where one of the connections is transverse (and so stable under small perturbations) while the other is structurally unstable, is called a Bykov cycle. A \emph{Bykov network} is a connected union of heteroclinic cycles, not necessarily in finite number. We refer to \cite{HS} for an overview of heteroclinic bifurcations and substantial information on the dynamics near different kinds of heteroclinic cycles and networks.

\subsection{Suspended horseshoe}
Given $(A, \lambda) \in [0, \varepsilon]^2$, suppose that there is a cross-section $\mathcal{S}_\lambda$ to the flow $\varphi_{(A, \lambda)}$ such that $\mathcal{S}_{(A, \lambda)}$ contains a compact set $\mathcal{K}_{(A, \lambda)}$ invariant by the first return map $\mathcal{F}_{(A, \lambda)}$ to $\mathcal{S}_{(A, \lambda)}$. Assume also that $\mathcal{F}_{(A, \lambda)}$ restricted to $\mathcal{K}_{(A, \lambda)}$ is conjugate to a full shift on a finite alphabet. Then the \emph{suspended horseshoe associated to $\mathcal{K}_{(A, \lambda)}$} is the flow-invariant set $\widetilde{\mathcal{K}_{(A, \lambda)}}=\{\varphi_\lambda(t,x)\,:\,t\in\RR,\,x\in \mathcal{K}_{(A, \lambda)}\}.$

\subsection{SRB measure}
Given an attracting set ${\Omega}$ for a continuous map $R: \mathcal{M} \rightarrow \mathcal{M}$ of a compact manifold $ \mathcal{M}$, consider the Birkhoff average with respect to the continuous function $T:  \mathcal{M} \rightarrow \RR$ on the $R$-orbit starting at $x\in  \mathcal{M}$:
\begin{equation}
\label{limit1}
L(T, x)=\lim_{n\in \NN} \quad \frac{1}{n} \sum_{i=0}^{n-1} T \circ R^i(x).
\end{equation}

Suppose that, for Lebesgue almost all points $x\in \textbf{B}({\Omega})$,  the limit \eqref{limit1} exists and is independent on $x$. Then $L$ is a continuous linear functional in the set of continuos maps from  $\mathcal{M}$ to $\RR$ (denoted by $C(\mathcal{M}, \RR)$). By the Riesz Representation Theorem, it defines a unique probability measure $\mu$ such that:
\begin{equation}
\label{limit2}
\lim_{n\in \NN} \quad \frac{1}{n} \sum_{i=0}^{n-1} T \circ R^i(x) = \int_{\Omega} T \, d\mu
\end{equation}
for all $T\in C(\mathcal{M}, \RR)$ and for Lebesgue almost all points $x\in \textbf{B}({\Omega})$.  If there exists an ergodic measure $\mu$ supported in ${\Omega}$ such that \eqref{limit2} is satisfied for all continuous maps $T\in C(\mathcal{M}, \RR)$ for Lebesgue almost all points $x\in \textbf{B}({\Omega})$, where $\textbf{B}({\Omega})$ has positive Lebesgue measure, then $\mu$ is called a SRB measure and ${\Omega}$ is a SRB attractor.

\end{document}